\documentclass[a4paper]{amsart}
\usepackage{microtype}

\usepackage{amsmath,amstext,amssymb,mathrsfs,amscd,amsthm,indentfirst, bbm}
\usepackage{amsfonts}

\usepackage{booktabs}
\usepackage{verbatim}
\usepackage{enumerate}
\usepackage{graphicx}
\usepackage{textcomp}

\newcommand{\bC}{\mathbb{C}}
\newcommand{\bD}{\mathbb{D}}

\newcommand{\bN}{\mathbb{N}}

\newcommand{\bP}{\mathbb{P}}
\newcommand{\bQ}{\mathbb{Q}}
\newcommand{\bR}{\mathbb{R}}
\newcommand{\bS}{\mathbb{S}}

\newcommand{\bZ}{\mathbb{Z}}

\newcommand{\gA}{\mathbf{A}}
\newcommand{\gV}{\mathbf{V}}
\newcommand{\gL}{\mathbf{L}}
\newcommand{\gP}{\mathbf{P}}

\newcommand\lra{\longrightarrow}

\newcommand\trf{\mathrm{trf}}
\newcommand\Diff{\mathrm{Diff}}

\newcommand\Ker{\operatorname*{Ker}}

\newcommand{\hcoker}{/\!\!/}

\renewcommand{\epsilon}{\varepsilon}

\newcommand{\Kdim}{\mathrm{Kdim}}

\newcommand{\HBmod}{H_B\bQ\text{-}\mathsf{mod}}
\newcommand{\HPBmod}{HP_B\bQ\text{-}\mathsf{mod}}

\mathchardef\ordinarycolon\mathcode`\:
\mathcode`\:=\string"8000
\begingroup \catcode`\:=\active
  \gdef:{\mathrel{\mathop\ordinarycolon}}
\endgroup

\usepackage[all]{xy}
\CompileMatrices

\allowdisplaybreaks

\usepackage{amsthm}
\theoremstyle{plain}
\newtheorem{MainThm}{Theorem}
\newtheorem{MainCor}[MainThm]{Corollary}

\newtheorem{theorem}{Theorem}[section]

\newtheorem{proposition}[theorem]{Proposition}
\newtheorem{lemma}[theorem]{Lemma}

\newtheorem{corollary}[theorem]{Corollary}

\theoremstyle{definition}

\newtheorem{example}[theorem]{Example}

\newtheorem{construction}[theorem]{Construction}

\theoremstyle{remark}
\newtheorem{obs}[theorem]{Observation}
\newtheorem{remark}[theorem]{Remark}
\newtheorem*{remark*}{Remark}

\numberwithin{equation}{section}

\usepackage[hypertexnames=false]{hyperref}

\SelectTips{eu}{10}

\title[Tautological rings]{Some phenomena in\\tautological rings of manifolds}

\author{Oscar Randal-Williams}
\email{o.randal-williams@dpmms.cam.ac.uk}
\address{Centre for Mathematical Sciences\\
Wilberforce Road\\
Cambridge CB3 0WB\\
UK}
\begin{document}
\begin{abstract}
We prove several basic ring-theoretic results about tautological rings of manifolds $W$, that is, the rings of generalised Miller--Morita--Mumford classes for fibre bundles with fibre $W$. Firstly we provide conditions on the rational cohomology of $W$ which ensure that its tautological ring is finitely-generated, and we show that these conditions cannot be completely relaxed by giving an example of a tautological ring which fails to be finitely-generated in quite a strong sense. Secondly, we provide conditions on torus actions on $W$ which ensure that the rank of the torus gives a lower bound for the Krull dimension of the tautological ring of $W$. Lastly, we give extensive computations in the tautological rings of $\mathbb{CP}^2$ and $S^2 \times S^2$.
\end{abstract}
\maketitle

\section{Introduction}

\subsection{Recollections on tautological rings}

A smooth fibre bundle $\pi : E \to B$ with closed $d$-dimensional fibre $W$ equipped with an orientation of the vertical tangent bundle $T_\pi E$ has characteristic classes defined as follows. For each characteristic class $c \in H^k(BSO(d))$ of oriented $d$-dimensional vector bundles, we may form
$$\kappa_c(\pi) := \int_\pi c(T_\pi E) \in H^{k-d}(B),$$
the \emph{generalised Mumford--Morita--Miller class} (or \emph{$\kappa$-class}) associated to $c$, by evaluating $c$ on the vector bundle $T_\pi E$ and integrating the result along the fibres of the map $\pi$. This construction may in particular be applied to the universal such fibre bundle, whose base space is the classifying space $B\Diff^+(W)$ of the topological group of orientation-preserving diffeomorphisms of $W$, to give universal characteristic classes $\kappa_c \in H^{*}(B\Diff^+(W))$. If $c$ has degree $d$ then $\kappa_c$ is a degree zero cohomology class, and may be identified with the characteristic number $\int_W c(TW)$ of $W$.

If we work in cohomology with rational coefficients then $H^*(BSO(d);\bQ)$ is generated by the Pontrjagin and Euler classes, and in this case we define the \emph{tautological ring}
$$R^*(W) \subset H^{*}(B\Diff^+(W);\bQ)$$
to be the subring generated by all classes $\kappa_c$. Our goal is to describe some quantitative and qualitative properties of these rings, for certain manifolds $W$.

Before doing so, we introduce some variants. The topological group $\Diff^+(W, *)$ of diffeomorphisms of $W$ which fix a marked point $* \in W$ has a homomorphism to $GL^+_d(\bR)$ by sending a diffeomorphism $\varphi$ to its differential $D\varphi_*$ at the marked point. On classifying spaces this gives a map
$$s : B\Diff^+(W, *) \lra BGL^+_d(\bR) \simeq BSO(d)$$
and for each $c \in H^*(BSO(d);\bQ)$ we may also form $s^*c \in H^*(B\Diff^+(W, *);\bQ)$. We let the \emph{tautological ring fixing a point} $R^*(W, *) \subset H^*(B\Diff^+(W, *);\bQ)$ be the subring generated by all the classes $\kappa_c$ and $s^*c$.

Finally, if $B\Diff^+(W, D^d)$ is the classifying space of the group of diffeomorphisms of $W$ which are the identity near a marked disc $D^d \subset W$, then we let the \emph{tautological ring fixing a disc} $R^*(W, D^d) \subset H^*(B\Diff^+(W, D^d);\bQ)$ be the subring generated by all the classes $\kappa_c$. The inclusions of diffeomorphism groups
$$B\Diff^+(W) \longleftarrow B\Diff^+(W, *) \longleftarrow B\Diff^+(W, D^d)$$
therefore yield $\bQ$-algebra homomorphisms
$$R^*(W) \lra R^*(W, *) \lra R^*(W, D^d)$$
whose composition is surjective.

These rings have been studied by Grigoriev \cite{grigoriev-relations}, and by Galatius, Grigoriev, and the author \cite{galagriran-characteristic}, mainly for the manifolds $W = \#^g S^n \times S^n$ with $n$ odd. This is the natural generalisation of the case of oriented surfaces, i.e.\ $n=1$, which has been studied in great detail: see e.g.\, \cite{Mumford, Looijenga, Faber, Morita}. Our purpose here is to explain to what extent those results apply to more general manifolds. We will only consider even-dimensional manifolds. For odd-dimensional manifolds the classes $\kappa_c$ have odd degree and so anticommute and are nilpotent, and tautological rings in this situation seem to have a different flavour.

\subsection{Finiteness}

Our first result concerns conditions under which the rings $R^*(W)$ and $R^*(W,*)$ are suitably finite.

\begin{MainThm}\label{mainThm:A}
Let $W$ be a closed smooth oriented $2n$-manifold, and assume that either
\begin{enumerate}[(H1)]
\item $H^*(W;\bQ)$ is non-zero only in even degrees, or

\item $H^*(W;\bQ)$ is non-zero only in degrees $0$, $2n$ and odd degrees, and $\chi(W) \neq 0$.
\end{enumerate}
Then
\begin{enumerate}[(i)]
\item $R^*(W)$ is a finitely-generated $\bQ$-algebra, and

\item $R^*(W, *)$ is a finitely-generated $R^*(W)$-module.
\end{enumerate}
\end{MainThm}

The result under hypothesis (H2) generalises a theorem of Grigoriev \cite{grigoriev-relations}, and proceeds by establishing the same basic source of relations among $\kappa$-classes found by Grigoriev. In the case $2n=2$ this source of relations had been established by the author \cite{randal-wil-relations}, using ideas of Morita \cite{morita-families, morita-families2}. As the later results of \cite{grigoriev-relations} and the results of \cite{galagriran-characteristic} are deduced almost entirely from this basic source of relations, the same results largely follow assuming only hypothesis (H2). For example, for $g > 1$, $k$ odd, and $n \geq k$, it follows that
$$\bQ[\kappa_{ep_1}, \kappa_{ep_2}, \ldots, \kappa_{ep_{n-1}}] \lra R^*(\#^g S^k \times S^{2n-k})/\sqrt{0}$$
is surjective, which was obtained in \cite{galagriran-characteristic} only in the case $k=n$. We give details of this in Section \ref{sec:OddCohom}. 

The result under hypothesis (H1) is entirely new and its method of proof is novel. We consider a fibre bundle $W \to E \overset{\pi}\to B$ as determining a parametrised spectrum over $B$, and hence its rational ``cochains" as giving a parametrised $H\bQ$-module spectrum over $B$. We then use the notion of {Schur-finiteness} from the theory of motives to obtain a Cayley--Hamilton-type trace identity for endomorphisms of this cochain object, which establishes concrete relations among $\kappa$-classes. Later we shall describe some explicit calculations done using these relations.

\subsection{Krull dimension}

Our second main result is a general technique, continuing on from our work with Galatius and Grigoriev \cite[\S 4]{galagriran-characteristic}, for estimating the Krull dimension (for which we write $\Kdim$) of the rings $R^*(W)$ from below in terms of torus actions on $W$. The general statement is Theorem \ref{thm:main}, but the hypotheses of that theorem are somewhat involved: we state here one of its corollaries with hypotheses which are easy to verify.

\begin{MainCor}\label{mainCor:B}
Let a $k$-torus $T$ act effectively on $W$, and suppose that either
\begin{enumerate}[(i)]
\item $\chi(W) \neq 0$ and the fixed set $W^T$ is connected, or

\item the fixed set $W^T$ is discrete and non-empty.
\end{enumerate}
Then $\Kdim(R^*(W)) \geq k$.
\end{MainCor}

For example, if $W^{2n}$ is a quasitoric manifold then case (ii) gives the estimate $\Kdim(R^*(W)) \geq n$. As another example, if $W = \#^g S^n \times S^n$ with $n$ odd then it is a consequence of the localisation theorem in  equivariant cohomology (which we shall discuss in Section \ref{sec:Loc}) that \emph{any} torus action on $W$ has connected fixed set, so by case (i) restricting the $SO(n) \times SO(n)$-action on $W$ constructed in \cite[\S 4]{galagriran-characteristic} to a maximal torus (which has rank $n-1$) we obtain $\Kdim(R^*(W)) \geq n-1$ for $g > 1$, which recovers the calculation of that paper. This example admits many variants: the construction of \cite[\S 4]{galagriran-characteristic} can be easily modified to give a $SO(k) \times SO(2n-k)$-action on $\#^g S^k \times S^{2n-k}$, so for any odd $k$ and any $n$ we have
$$\Kdim(R^*(\#^g S^k \times S^{2n-k})) \geq n-1.$$
We shall say more about this example in Section \ref{sec:OddCohom}. 

\subsection{Examples}

In the last section of the paper we exhibit several phenomena in tautological rings by calculations for specific manifolds. The following result is complementary to Theorem \ref{mainThm:A}, and shows that the hypotheses of that theorem cannot be completely removed.

\begin{MainThm}\label{mainThm:C}
There are closed smooth manifolds $W$ for which $R^*(W)/\sqrt{0}$ is not finitely-generated as a $\bQ$-algebra. There are examples of any dimension $4k+2 \geq 6$, and in dimensions $4k+2 \geq 14$ such manifolds can also be assumed to be simply-connected.
\end{MainThm}

To show the effectiveness of the relations between $\kappa$-classes arising in the proof of Theorem \ref{mainThm:A}, we apply them to the simplest manifold whose tautological ring is not yet known, namely $\bC\bP^2$. These relations, along with relations associated to the Hirzebruch $\mathcal{L}$-classes coming from index theory, give the following.

\begin{MainThm}\label{mainThm:D}
The ring $R^*(\bC\bP^2)$ has Krull dimension 2. The ring $R^*(\bC\bP^2, D^4)$ is a vector space of dimension at most 7 over $\bQ$. 
\end{MainThm}

In fact, we show that the ring $R^*(\bC\bP^2)/\sqrt{0}$ is equal to either
$$\bQ[\kappa_{p_1^2}, \kappa_{ep_1}, \kappa_{p_1^4}]/(4 \kappa_{p_1^2}-7 \kappa_{e p_1})\cap (\kappa_{p_1^2}-2 \kappa_{e p_1},316 \kappa_{e p_1}^3-343 \kappa_{p_1^4}),$$
whose variety is the union of a line and a plane, or
$$\bQ[\kappa_{p_1^2}, \kappa_{ep_1}, \kappa_{p_1^4}]/(4 \kappa_{p_1^2}-7 \kappa_{e p_1}),$$
whose variety is a plane. It would be interesting to determine which case occurs, and very interesting if it is the first case.

Finally, we give a calculation which shows that the lower bound of Corollary \ref{mainCor:B} is not always sharp. The 3-torus cannot act effectively on $S^2 \times S^2$, and yet

\begin{MainThm}\label{mainThm:E}
The ring $R^*(S^2 \times S^2)$ has Krull dimension 3 or 4.
\end{MainThm}

The lower bound on the Krull dimension comes from a 1-parameter family of 2-torus actions, to which the method of proof of Corollary \ref{mainCor:B} is applied. The upper bound comes from the relations between $\kappa$-classes which we found in the proof of Theorem \ref{mainThm:A}.

\subsection*{Acknowledgements} I am grateful to S{\o}ren Galatius for an enlightening discussion of the ideas in Section \ref{sec:fg} of this paper, and to Jens Reinhold and Dexter Chua for spotting several errors. I would also like to thank the anonymous referee for their useful suggestions. I was partially supported by EPSRC grant EP/M027783/1.

\section{Tautological relations and finite generation}\label{sec:fg}

\emph{Unless specified, all cohomology in this paper will be taken with $\bQ$ coefficients.}

\vspace{1ex}

In this section we describe some techniques for obtaining relations between tautological classes, which for some manifolds $W$ suffice to establish that $R^*(W)$ is finitely-generated. The techniques we will introduce are perhaps more important than any particular application that can be made, but Theorem \ref{mainThm:A} will be a consequence. 

\subsection{Integrality}

One consequence of conclusion (ii) of Theorem \ref{mainThm:A} is that $R^*(W,*)$ is integral over $R^*(W)$. In fact, this integrality statement implies the two finiteness statements, as follows.

\begin{proposition}\label{prop:Integral}
Suppose that $W$ is a closed smooth oriented $d$-manifold such that $R^*(W,*)$ is integral over $R^*(W)$. Then
\begin{enumerate}[(i)]
\item $R^*(W)$ is a finitely-generated $\bQ$-algebra, and

\item $R^*(W, *)$ is a finitely-generated $R^*(W)$-module.
\end{enumerate}
\end{proposition}

For the sake of clarity, we will first formulate and prove a purely algebraic statement of which this proposition is a consequence.

\begin{lemma}\label{lem:AlgIntegral}
Let $\pi : B \to E$ be a homomorphism of $\bQ$-algebras, $g : E \to B$ be a homomorphism of $B$-modules (where $E$ is made into an $B$-module via $\pi$), and $C \subset E$ be a finitely-generated subalgebra, generated by $\{c_i\}_{i \in I}$. Let $R \subset B$ be the subalgebra generated by $g(C)$. If each $c_i$ is integral over $\pi(R) \subset E$, with
$$c_i^{n_i} = \sum_{j=0}^{n_i-1} \pi(a_{i,j}) c_i^j$$
for some $a_{i,j} \in R$, then $R$ is generated by the finitely-many elements 
$$\{a_{i,j}\}_{i, j \in I} \cup \{ g(\prod c_i^{m_i}) \, \vert \, m_i < n_i\}.$$
\end{lemma}
\begin{proof}
By definition $R$ is generated by the elements $g(\prod c_i^{k_i})$, so we must show that these lie in the subring generated by the indicated elements. By assumption we may write $c_i^{k_i}$ as a $\bQ[\pi(a_{i,j})]$-linear combination of terms $c_i^j$ with $j < n_i$, and so we may write $\prod c_i^{k_i}$ as a $\bQ[\pi(a_{i,j})]$-linear combination of terms $\prod c_i^{m_i}$ with $m_i < n_i$. As $g$ is $B$- and hence $R$-linear, we may therefore write $g(\prod c_i^{k_i})$ as a $\bQ[a_{i,j}]$-linear combination of terms $g(\prod c_i^{m_i})$ with $m_i < n_i$. Thus $g(\prod c_i^{k_i})$ lies in the subring generated by the indicated elements.
\end{proof}

\begin{proof}[Proof of Proposition \ref{prop:Integral}]
The universal fibre bundle with fibre $W$ may be identified with the natural projection
$$p : B\Diff^+(W, *)  \cong E\Diff^+(W) \times_{\Diff^+(W)} W \lra B\Diff^+(W).$$
This gives a $\bQ$-algebra homomorphism
$$p^* : H^*(B\Diff^+(W);\bQ) \lra H^*(B\Diff^+(W, *);\bQ)$$
by pullback and an $H^*(B\Diff^+(W);\bQ)$-module homomorphism
$$p_!: H^*(B\Diff^+(W, *);\bQ) \to H^{*-d}(B\Diff^+(W);\bQ)$$
by fibre integration. As we have described in the introduction, taking the differential at the marked point gives a map $s : B\Diff^+(W, *) \to BGL_d^+(\bR) \simeq BSO(d)$, which classifies the vertical tangent bundle of the universal fibre bundle $p$.

Applying Lemma \ref{lem:AlgIntegral} with $B = H^*(B\Diff^+(W);\bQ)$, $E = H^*(B\Diff^+(W, *);\bQ)$, $\pi = p^*$, $g = p_!$, and $C = \mathrm{Im}(s^*)$ (using that $H^*(BSO(d);\bQ)$ is finitely-generated and that $C \subset R^*(W, *)$ so by assumption consists of elements which are integral over $R = R^*(W)$) shows that $R^*(W) \subset H^*(B\Diff^+(W);\bQ)$ is finitely-generated, proving the first part.

For the second part, as $H^*(BSO(d))$ is a finitely-generated $\bQ$-algebra, we know that $R^*(W, *)$ is a finitely-generated $R^*(W)$-algebra, so under the integrality assumption it follows that $R^*(W, *)$ is in fact finitely-generated as a $R^*(W)$-module.
\end{proof}

Thus in order to prove Theorem \ref{mainThm:A} we shall actually show that $R^*(W,*)$ is integral over $R^*(W)$.

\subsection{Outline}

To motivate the proof of Theorem \ref{mainThm:A} let us first explain its proof under hypothesis (H1) and an additional assumption: that the universal smooth oriented fibre bundle $W \to E \overset{\pi}\to B = B\Diff^+(W)$ satisfies the Leray--Hirsch property in rational cohomology, i.e.\ that $\pi_1(B)$ acts trivially on $H^*(W)$ and the Serre spectral sequence for $\pi: E \to B$ collapses. (The proof of Theorem \ref{mainThm:A} under hypothesis (H1) is a technical device which allows the following argument to be made without this additional assumption.)

%We first develop some algebra. 
Under this assumption, $H^*(E)$ is a free finitely-generated $H^*(B)$-module, say with basis $\bar{x}_1, \ldots, \bar{x}_k \in H^*(E)$ lifting a basis $x_1, \ldots, x_k$ for $H^*(W)$. Furthermore, as $W$ has all its cohomology in even degrees, $H^{ev}(E)$ is a free finitely-generated module over the commutative ring $H^{ev}(B)$, with basis the $\bar{x}_i$. For $x \in H^{ev}(E)$ the map
$$- \cdot x : H^{ev}(E) \lra H^{ev}(E)$$
is a $H^{ev}(B)$-module map, so has a characteristic polynomial $\chi_x(z) \in H^{ev}(B)[z]$, and by the Cayley--Hamilton theorem (for finite modules over a commutative ring, alias the determinantal trick) we have $\chi_x(x)=0 \in H^{ev}(E)$. Furthermore, the coefficients of the characteristic polynomial $\chi_x(z)$ may be expressed as polynomials in the elements
$$\mathrm{Tr}(- \cdot x^i : H^{ev}(E) \to H^{ev}(E)) \in H^{ev}(B),$$
which make sense as $H^{ev}(E)$ is a finite free $H^{ev}(B)$-module. The following lemma relates such traces to fibre-integration and the Euler class of the vertical tangent bundle.

\begin{lemma}\label{lem:TraceIsTransfer}
For any $x \in H^{ev}(E)$ we have
$$\mathrm{Tr}(- \cdot x : H^{ev}(E) \to H^{ev}(E)) = \int_\pi e(T_\pi E) \cdot x  \in H^{ev}(B).$$
\end{lemma}

We apply the above discussion to $x = c(T_\pi E)$ for $c \in H^*(BSO(2n))$ a characteristic class of oriented $2n$-dimensional vector bundles. Then the polynomial $\chi_x(z)$ is monic, has coefficients in the subring generated by the $\kappa_{ec^i} = \int_\pi e(T_\pi E) \cdot c(T_\pi E)^i$, and satisfies $\chi_x(c(T_\pi E))=0$. Thus we deduce that $c(T_\pi E) = s^*c$ is integral over $R^*(W)$ and hence that $R^*(W,*)$ is integral over $R^*(W)$. Theorem \ref{mainThm:A} in the case we are considering follows by applying Proposition \ref{prop:Integral}. It remains to prove this lemma.

\begin{proof}[Proof of Lemma \ref{lem:TraceIsTransfer}]
Rational cohomology classes are determined by their evaluations against rational homology classes, and any rational homology class is carried on a map from a smooth oriented manifold. So we may assume that $\pi: E \to B$ is a fibre bundle over a smooth oriented manifold, still satisfying the Leray--Hirsch property.

The pairing
\begin{align*}
\langle -, - \rangle : H^*(E) \otimes_{H^*(B)} H^*(E) &\lra H^{*-d}(B)\\
a \otimes b &\longmapsto \int_\pi a\cdot b
\end{align*}
is non-singular, as in the basis $\bar{x}_i$ its matrix $X$ agrees modulo the ideal $H^{*>0}(B) $ of $H^*(B)$ with that of the intersection form of $W$ in the basis $x_i$, so $\det(X)\in H^*(B)$ is a unit modulo $H^{*>0}(B)$, and hence is a unit as the ideal $H^{*>0}(B)$ is nilpotent. Let us write $\bar{x}_i^\vee$ for the dual $H^*(B)$-module basis of $H^*(E)$ with respect to this pairing, characterised by $\langle \bar{x}_i, \bar{x}_j^\vee \rangle = \delta_{ij}$. Then for any $x \in H^{ev}(E)$ we have
$$\mathrm{Tr}(- \cdot x : H^{ev}(E) \to H^{ev}(E)) = \sum_i \langle \bar{x}_i \cdot x, \bar{x}_i^\vee \rangle=  \int_\pi\left(\left(\sum_i \bar{x}_i \cdot \bar{x}_i^\vee \right) \cdot x \right),$$
so to establish the claimed formula we must show that $\sum_i \bar{x}_i \cdot \bar{x}_i^\vee = e(T_\pi E) \in H^{ev}(E)$.

The diagonal map $\Delta : E \to E \times_B E$ is a map of smooth oriented manifolds, whose normal bundle is identified with $T_\pi E$. Thus the Euler class $e(T_\pi E) \in H^d(E)$ may be described as $\Delta^* \Delta_!(1)$. It is therefore enough to show that 
$$\Delta_!(1) = \sum_i \bar{x}_i \otimes \bar{x}_i^\vee \in H^*(E \times_B E) = H^*(E) \otimes_{H^*(B)} H^*(E).$$
This is the parametrised analogue of the classical formula \cite[Theorem 11.11]{milnostash-characteristic} for the Poincar{\'e} dual of the diagonal, and we shall prove it in the same way. For any $b \in H^*(B)$ we calculate
\begin{align*}
\int_{E \times_B E} \Delta_!(1)  \cdot ((b \cdot \bar{x}_j^\vee) \otimes \bar{x}_k) &= \int_E b \cdot \bar{x}_j^\vee \cdot \bar{x}_k \\
 &= \int_B b \int_\pi \bar{x}_j^\vee \cdot \bar{x}_k = \delta_{jk} \int_B b
\end{align*}
and 
\begin{align*}
\int_{E \times_B E} \left(\sum_i \bar{x}_i \otimes \bar{x}_i^\vee \right) \cdot ((b \cdot \bar{x}_j^\vee) \otimes \bar{x}_k) &= \sum_i \int_{E \times_B E} (b \cdot \bar{x}_j^\vee \cdot \bar{x}_i) \otimes ( \bar{x}_k \cdot \bar{x}_i^\vee) \\
 &= \sum_i \int_B b \int_{\pi \times_B \pi} (\bar{x}_j^\vee \cdot \bar{x}_i) \otimes ( \bar{x}_k \cdot \bar{x}_i^\vee) \\
&= \sum_i \delta_{ij} \delta_{ki} \int_B b = \delta_{jk} \int_B b.
\end{align*}
As the classes $(b \cdot \bar{x}_j^\vee) \otimes \bar{x}_k$ generate $H^*(E \times_B E)$ as a $\bQ$-module, it follows from Poincar{\'e} duality for $E \times_B E$ that $\Delta_!(1) = \sum_i \bar{x}_i \otimes \bar{x}_i^\vee$, as required.
\end{proof}

\subsection{Parametrised spectra and Schur functors}

The technical device we shall use to attempt the argument of the previous section without the Leray--Hirsch assumption is to consider a fibre bundle as a parametrised manifold over its base, and make the argument in the parametrised setting. In order to do so, we shall suppose that $B$ is a connected CW-complex, and work in a symmetric monoidal category $(\mathsf{Sp}_{/B}, \wedge_B, S^0_B)$ of parametrised spectra over $B$. For concreteness we take the category developed by May--Sigurdsson \cite{MaySig}\footnote{However, our arguments are not model-dependent and can be applied in the $\infty$-categorical formalism of Ando, Blumberg, Gepner, Hopkins, and Rezk \cite{ABGHR1, ABG}, and presumably even in more na{\"i}ve models of parametrised spectra.}. We will write $r : B \to \{*\}$ for the unique map; then, as $\mathsf{Sp}_{/*} = \mathsf{Sp}$, by \cite[Theorem 11.4.1]{MaySig} there are right and left adjoint functors
$$r^* : \mathsf{Sp} \lra \mathsf{Sp}_{/B} \quad \text{ and }\quad r_! : \mathsf{Sp}_{/B} \lra \mathsf{Sp},$$
(apart from this map, the notation $(-)_!$ will always denote Gysin maps). The functor $r^*$ is strong monoidal.

Our argument applies more generally than to oriented fibre bundles: for now, we let $\pi : E \to B$ be a Hurewicz fibration (later we will add a finiteness hypothesis to the fibres of $\pi$). This defines a parametrised spectrum $\Sigma^\infty_B E \in \mathsf{Sp}_{/B}$; we shall abuse notation and continue to call it $E$. Note that $r_!(E) \in \mathsf{Sp}$ is the suspension spectrum $\Sigma^\infty E_+$.

The ring spectrum $H\bQ$ has a 2-periodic version
$$HP \bQ = \bigvee_{i \in \bZ} \Sigma^{2i} H\bQ,$$
and we write $\pi_*(HP \bQ) = \bQ[t^{\pm 1}]$ with $t \in \pi_2(HP\bQ)$. The constant parametrised spectra $H_B\bQ := r^*(H\bQ)$ and $HP_B\bQ := r^*(HP\bQ)$ define ring objects in $\mathsf{Sp}_{/B}$, and the main objects we will consider are the function objects 
$$C := F_B(E, H_B\bQ) \quad\text{ and }\quad CP := F_B( E, HP_B\bQ).$$
These are again ring objects, using the fibrewise diagonal map on $E$ and the multiplication on $H_B\bQ$ and $HP_B\bQ$; we write $\mu$ for the multiplication on either object. The map $E \to *$ gives ring maps $H_B \bQ \to C$ and $HP_B\bQ \to CP$, making them $H_B\bQ$- and $HP_B\bQ$-modules respectively. 

Let us write $(\HBmod, \otimes, H_B\bQ)$ for the homotopy category of $H_B\bQ$-module spectra, with derived smash product of $H_B\bQ$-modules as the symmetric monoidal structure and unit $H_B\bQ$; similarly write $(\HPBmod, \otimes, HP_B\bQ)$ for the homotopy category of $HP_B\bQ$-module spectra. We have $C \in \HBmod$ and $CP \in \HPBmod$, and we can calculate
\begin{align*}
[\Sigma^d H_B\bQ, C]_{\HBmod} &= [\Sigma^d S_B^0, C]_{\mathsf{Sp}_{/B}} \\
&= [E, \Sigma^{-d}H_B \bQ]_{\mathsf{Sp}_{/B}} =[E, r^*(\Sigma^{-d}H \bQ)]_{\mathsf{Sp}_{/B}}\\
&= [\Sigma^\infty E_+, \Sigma^{-d} H\bQ]_{\mathsf{Sp}}= H^{-d}(E)
\end{align*}
and
\begin{align*}
[HP_B\bQ, CP]_{\HPBmod} &= [S_B^0, CP]_{\mathsf{Sp}_{/B}} \\
&= [ E, HP_B \bQ]_{\mathsf{Sp}_{/B}} =  [ E, r^*(HP \bQ)]_{\mathsf{Sp}_{/B}}\\
&= \left[\Sigma^\infty E_+, \bigvee_{i \in \bZ} \Sigma^{2i} H\bQ\right]_{\mathsf{Sp}} = \bigoplus_{i \in \bZ} H^{2i}(E).
\end{align*}
Both $\HBmod$ and $\HPBmod$ are $\bQ$-linear tensor categories (i.e.\ categories enriched in $\bQ$-modules, equipped with a symmetric monoidal structure which is an enriched functor) which are idempotent complete (the retract associated to an endomorphism $e : X \to X$ which is idempotent up to homotopy may be taken to be the homotopy colimit of a diagram $X \overset{e}\to X\overset{e}\to X \overset{e}\to \cdots$ of modules over the appropriate ring object).

We must now recall a little representation theory of symmetric groups; we need nothing beyond Lecture 4 of \cite{FultonHarris}. To each partition $\lambda$ of $n$ there is associated an irreducible representation $S^\lambda$ of $\Sigma_n$, with character $\chi_\lambda$. This character takes rational (in fact, integer) values, so we may form the element
$$d_\lambda := \frac{\dim S^\lambda}{n!} \sum_{\sigma \in \Sigma_n} \chi_{\lambda}(\sigma) \cdot \sigma \in \bQ[\Sigma_n],$$
which is central (as $\chi_\lambda$ is a class function) and idempotent (the coefficient $\tfrac{\dim S^\lambda}{n!}$ is chosen to make this so). For any object $X$ in a $\bQ$-linear tensor category $(\mathsf{D}, \otimes, \mathbbm{1})$, the action of the $n$th symmetric group $\Sigma_n$ on $X^{\otimes n}$ yields a map of $\bQ$-algebras
$$e: \bQ[\Sigma_n] \lra \mathrm{Hom}_\mathsf{D}(X^{\otimes n}, X^{\otimes n}),$$
so $e(d_\lambda)$ is an idempotent endomorphism of $X^{\otimes n}$ in $\mathsf{D}$; if $\mathsf{D}$ is idempotent complete then we write $S_\lambda(X)$ for the corresponding retract of $X^{\otimes n}$ in $\mathsf{D}$: this defines the \emph{Schur functor} $S_\lambda(-)$ on $\mathsf{D}$. In this paper the trivial and sign representations will play the most prominent role, and we write 
$$\wedge^n X := S_{(1^n)}(X) \quad \text{ and } \quad \mathrm{Sym}^n(X) := S_{(n)}(X),$$
or, if we wish to emphasise the ambient category, $\wedge^n_\mathsf{D}$ and $\mathrm{Sym}^n_\mathsf{D}$.

The categories $\HBmod$ and $\HPBmod$ are idempotent complete $\bQ$-linear tensor categories, so there are defined Schur functors on both categories. Furthermore, let us write $(\mathsf{V}_\bQ, \otimes_\bQ, \bQ)$ for the symmetric monoidal category of graded $\bQ$-modules, and $(\mathsf{V}_{\bQ[t^{\pm 1}]}, \otimes_{\bQ[t^{\pm 1}]}, \bQ[t^{\pm 1}])$ for the symmetric monoidal category of graded $\bQ[t^{\pm 1}]$-modules (where $t$ has degree 2). These are also idempotent complete $\bQ$-linear tensor categories. Taking homotopy groups defines functors
\begin{align*}
\pi_*(-) : H\bQ\text{-}\mathsf{mod} &\lra \mathsf{V}_\bQ\\
\pi_*(-) : HP\bQ\text{-}\mathsf{mod} &\lra \mathsf{V}_{\bQ[t^{\pm 1}]}
\end{align*}
which are strong monoidal (by the K{\"u}nneth theorem, as every graded $\bQ$- or $\bQ[t^{\pm 1}]$-module is free). Taking derived homotopy fibres at $b \in B$ defines functors
\begin{align*}
(-)_b : \HBmod &\lra H\bQ\text{-}\mathsf{mod}\\
(-)_b : \HPBmod &\lra HP\bQ\text{-}\mathsf{mod}
\end{align*}
which are strong monoidal and reflect isomorphisms (by definition, cf.\ \cite[Definition 12.3.4]{MaySig}, as $B$ is assumed path-connected). Finally,
$$- \otimes_\bQ \bQ[t^{\pm 1}] : \mathsf{V}_\bQ \lra \mathsf{V}_{\bQ[t^{\pm 1}]}$$
is also strong monoidal. In particular, all of the above functors preserve Schur functors.

\begin{lemma}\label{lem:SchurKills}
Let $X \in \HBmod$ or $\HPBmod$ be such that for each fibre $X_b$ we have $S_{\lambda}(\pi_*(X_b))=0$. Then $S_{\lambda}(X) \simeq *$.
\end{lemma}
\begin{proof}
As taking homotopy groups preserves Schur functors, we have $\pi_*(S_\lambda(X_b)) = S_{\lambda}(\pi_*(X_b))$ which vanishes by assumption. Thus $S_\lambda(X_b) \simeq *$, so as taking derived fibres preserves Schur functors it follows that $S_\lambda(X)_b \simeq *$. Thus the map from $S_\lambda(X)$ to the terminal object is an equivalence on derived fibres, and hence an equivalence, as taking derived fibres reflects isomorphisms.
\end{proof}

\subsection{Duals, trace, and transfer}

We recall the framework of categorical traces, from \cite{DoldPuppe}. 
%We now pass to topology.  We wish to use the description of the Becker--Gottlieb transfer as a categorical trace as in \cite{DoldPuppe} and \cite[Section 15.2]{MaySig}. Let us recall the required framework. 
If $(\mathsf{C}, \otimes, \mathbbm{1})$ is a symmetric monoidal category and $X \in \mathsf{C}$ is an object, a \emph{strong dual} of $X$ is an object $X^\vee \in \mathsf{C}$ and morphisms
$$\epsilon : X^\vee \otimes X \lra \mathbbm{1} \quad\quad\quad \eta : \mathbbm{1} \lra X \otimes X^\vee$$
such that the compositions $(X \otimes \epsilon) \circ (\eta \otimes X)$ and $(\epsilon \otimes X^\vee)\circ(X^\vee \otimes \eta)$ are the identity maps of $X$ and $X^\vee$ respectively. If $f : X \to Y$ is a map of objects having strong duals, then the dual of $f$ is
$$f^\vee : Y^\vee \overset{Y^\vee \otimes \eta}\lra Y^\vee \otimes X \otimes X^\vee \overset{Y^\vee \otimes f \otimes X^\vee}\lra Y^\vee \otimes Y \otimes X^\vee \overset{\epsilon \otimes X^\vee}\lra X^\vee.$$

If $f : X \to X$ is an endomorphism of $X$, the \emph{trace of $f$} is the composition
$$\mathrm{Tr}(f) : \mathbbm{1} \overset{\eta}\lra X \otimes X^\vee \cong X^\vee \otimes X \overset{f^\vee \otimes X}\lra X^\vee \otimes X\overset{\epsilon}\lra \mathbbm{1}.$$
This agrees with the perhaps more obvious choice
$$\mathrm{Tr}(f) : \mathbbm{1} \overset{\eta}\lra X \otimes X^\vee \overset{f \otimes X^\vee}\lra X \otimes X^\vee \cong X^\vee \otimes X \overset{\epsilon}\lra \mathbbm{1},$$
but the first definition is that of \cite{DoldPuppe}. Generalising this more obvious choice, if $f : A \otimes X \to B \otimes X$ is a morphism then the \emph{trace of $f$ over $X$} is the composition
$$\mathrm{Tr}^X(f) : A \overset{A \otimes \eta}\lra A \otimes X \otimes X^\vee \overset{f \otimes X^\vee}\lra B \otimes X \otimes X^\vee \cong B \otimes X^\vee \otimes X \overset{B \otimes \epsilon}\lra B.$$

%It is easy to verify that $\mathrm{Tr}(f^\vee) = \mathrm{Tr}(f)$. 
If $X$ is in addition equipped with a comultiplication $d : X \to X \otimes X$ then the \emph{transfer of $f$} is
$$\tau(f) : \mathbbm{1} \overset{\eta}\lra X \otimes X^\vee \cong X^\vee \otimes X \overset{f^\vee \otimes d}\lra X^\vee \otimes X \otimes X\overset{\epsilon \otimes X}\lra X.$$

First, consider the symmetric monoidal category given by the homotopy category of $(\mathsf{Sp}_{/B}, \wedge_B, S^0_B)$. 

\begin{lemma}\label{lem:EDualisable}
If $\pi : E \to B$ is a Hurewicz fibration and its fibre has the homotopy type of a finite CW-complex, then its associated parametrised spectrum $\Sigma^\infty_B E$ is a strongly dualisable object in the homotopy category of parametrised spectra.
\end{lemma}
\begin{proof}
This follows from Theorem 15.1.1 of \cite{MaySig}.
%The map $\pi \sqcup \mathrm{Id}_B : E \sqcup B \to B$ with the obvious section is an ex-fibration over a finite CW complex and is well-based. Thus it is an ``ex-fibration" in the sense given at the bottom of page 112 of \cite{BGFib}, and hence $E \in \mathsf{Ho}(\mathsf{Sp}_{/B})$ is strongly dualisable by \cite[Theorem 4.7]{BGFib}.
\end{proof}

Suppose then that $\pi : E \to B$ is a Hurewicz fibration and its fibre has the homotopy type of a finite CW-complex. Recall that we abuse notation by writing $E$ for $\Sigma^\infty_B E$. The fibrewise suspension of the fibrewise diagonal map $\Delta : E \to E \times_B E$ gives a comultiplication on the object $E$, we may thus form 
$$\trf_\pi = \tau(\mathrm{Id}_{E}) : S_B^0 \lra E.$$
On applying $r_! : \mathsf{Sp}_{/B} \to \mathsf{Sp}$ this gives a map of spectra $\Sigma^\infty B_+ \to \Sigma^\infty E_+$, which on cohomology gives a map
$$\trf^*_\pi : H^*(E) \lra H^*(B),$$
the \emph{Becker--Gottlieb transfer}. See \cite{BGFib} or \cite[Section 15.3]{MaySig} for this construction of the Becker--Gottlieb transfer. When $\pi : E \to B$ is an oriented smooth fibre bundle, by \cite[Theorem 4.3]{beckegottl-transfer} we have the identity
\begin{equation}\label{eq:TrfEuler}
\trf^*_\pi(-) = \int_\pi e(T_\pi E) \cdot -  : H^*(E) \lra H^*(B).
\end{equation}
In particular for $c \in H^*(BSO(2n))$ we have that $\trf^*_\pi(c(T_\pi E)) = \kappa_{ec}$ is a tautological class.

Let us now consider the symmetric monoidal categories $(\HBmod, \otimes, H_B\bQ)$ and $(\HPBmod, \otimes, HP_B\bQ)$.

\begin{corollary}\label{cor:CCPDualisable}
If $\pi: E \to B$ is a Hurewicz fibration and its fibre has the homotopy type of a finite CW-complex, then $C = F_B(E, H_B\bQ)$ is a dualisable object of $\HBmod$, and $CP = F_B(E, HP_B\bQ)$ is a dualisable object of $\HPBmod$.
\end{corollary}
\begin{proof}
The functor
$$F_B(-, H_B\bQ) : \mathsf{Ho}(\mathsf{Sp}_{/B}) \lra \HBmod$$
has a monoidality given by the adjoint of the morphism
$$X \wedge_B Y \wedge_B (F_B(X, H_B\bQ) \wedge_{H_B\bQ} F_B(Y, H_B\bQ)) \lra H_B\bQ \wedge_B H_B\bQ \lra H_B\bQ$$
given by evaluation and product. This is a strong monoidality: the induced morphism
$$F(X_b, H\bQ) \wedge_{H\bQ} F(Y_b, H\bQ) \lra F(X_b \wedge Y_b, H\bQ)$$
on derived fibres is a weak equivalence (by the K{\"u}nneth theorem, as every $\pi_*(H\bQ)=\bQ$-module is free). As $E \in \mathsf{Ho}(\mathsf{Sp}_{/B})$ is strongly dualisable by Lemma \ref{lem:EDualisable}, so is $C = F_B(E, H_B\bQ)$, because strong monoidal functors preserve (strong) duals. The argument for $CP$ is identical.
\end{proof}

\subsection{Schur-finiteness and trace identities}\label{sec:TraceId}

Deligne has introduced \cite[\S 1]{Deligne} the notion of \emph{Schur-finiteness} of an object $X$ in an idempotent complete $\bQ$-linear tensor category to be the property that $S_\lambda(X)$ is trivial for some partition $\lambda \vdash n$. In this section we consider this notion applied to the category $(\HPBmod, \otimes, HP_B\bQ)$ and so consider an $HP_B\bQ$-module $X$ such that $S_{\lambda}(X) \simeq *$, and let us in addition suppose that $X$ is dualisable in $\HPBmod$. Let us write $X^\vee$ for the dual of $X$, with duality structure given by $\eta : HP_B\bQ \to X \otimes X^\vee$ and $\epsilon : X^\vee \otimes X \to HP_B\bQ$.

Given an endomorphism $f : X \to X$, we may form the endomorphism
$$X^{\otimes n} \overset{X \otimes f^{\otimes n-1}}\lra X^{\otimes n} \overset{e(d_\lambda)}\lra X^{\otimes n}$$
and take the trace over the last $(n-1)$ copies of $X$, i.e.\ apply the construction $\mathrm{Tr}^{X^{\otimes n-1}}(-)$ described in the previous section, to obtain an endomorphism of $X$. This endomorphism is null because the idempotent $e(d_{\lambda}) : X^{\otimes n} \to X^{\otimes n}$ factors through $S_\lambda(X)$ which is contractible by assumption. 

We now translate this into formulas. We have $d_\lambda = \frac{\dim S^\lambda}{n!} \sum_{\sigma \in \Sigma_n} \chi_{\lambda}(\sigma) \cdot \sigma$ so the essential calculation is to describe the endomorphism of $X$ obtained from $\sigma \circ ( X \otimes f^{\otimes n-1}) : X^{\otimes n} \to X^{\otimes n}$ by taking the trace over the last $(n-1)$ copies of $X$. This is a universal construction in idempotent complete $\bQ$-linear tensor categories, and has been worked out by Abramsky \cite[Proposition 3]{Abramsky}. In the notation of that proposition, one takes $A_1=B_1=X$ and $U_2= \cdots=U_n = X$, then $f_1=\mathrm{Id}_X$ and $f_2 = \cdots = f_n = f$, and $\pi = \sigma$. The trace of $\sigma \circ (X \otimes f^{\otimes n-1})$ over the last $(n-1)$ copies of $X$ is then given by
$$\left(\prod_{l \in \mathcal{L}(\sigma)} s_l \right) \cdot (p_{\sigma}^{-1} \circ  g_1).$$
Here $p_{\sigma}^{-1} = \mathrm{Id}_X$, and $g_1$ is given by composing the $f_i$ along the cycle in $\sigma$ starting at 1: as $f_1=\mathrm{Id}_X$, if this cycle is $(1, p_2, \ldots, p_k)$ then this gives $g_1 = f^{\circ k-1}$; $\mathcal{L}(\sigma)$ is the set of cycles in the permutation $\sigma$ which do not contain 1, and for such a cycle $l=(p_1, p_2, \ldots, p_k)$ we have $s_l := \mathrm{Tr}(f_{p_k} \circ \cdots \circ f_{p_1}) = \mathrm{Tr}(f^{\circ k})$.

Applying this discussion to the identity  $0 = \sum_{\sigma \in \Sigma_n} \chi_\lambda(\sigma) \cdot (\sigma \circ (X \otimes f^{\otimes n-1}))$ gives the identity
\begin{equation}\label{eq:trace1}
0 = \sum_{\sigma \in \Sigma_{n}} \chi_\lambda(\sigma) \cdot \mathrm{Tr}(f^{\circ l(\gamma_2)})\cdots \mathrm{Tr}(f^{\circ l(\gamma_{q(\sigma)})}) \cdot f^{\circ l(\gamma_1)-1} \in [X, X]_{\HPBmod}
\end{equation}
where $\sigma = \gamma_1 \cdot \gamma_2 \cdots \gamma_{q(\sigma)}$ is a decomposition into disjoint cycles, with $1$ being in the support of $\gamma_1$, and $l(\gamma_i)$ denotes the length of the cycle $\gamma_i$. 

We originally learnt this idea from the thesis of del Padrone \cite{dP} (see \cite[Proposition 2.2.4]{dP} for a closely related result).

\subsection{Proof of Theorem \ref{mainThm:A} under the first hypothesis}

In this case we will work with the periodic chains $CP$. For each $b \in B$ we have
\begin{align*}
\pi_0(CP_b) &= \bigoplus_{i \in \bZ}H^{-2i}(E_b) \cong \bigoplus_{i \in \bZ} H^{-2i}(W),\\
\pi_1(CP_b) &= \bigoplus_{i \in \bZ} H^{-2i-1}(E_b) \cong \bigoplus_{i \in \bZ} H^{-2i-1}(W) 
\end{align*}
and so by 2-periodicity we have an isomorphism $\pi_*(CP_b) \cong H^{-*}(W) \otimes_\bQ \bQ[t^{\pm 1}]$ of graded $\bQ[t^{\pm 1}]$-modules. Here the right-hand side is to be interpreted as the tensor product of graded $\bQ$-modules, which in degree $k$ is
$$\bigoplus_{-i+2n = k} H^{-i}(W) \otimes \bQ\{t^n\}.$$
Thus we have $\wedge^{\ell}_{\bQ[t^{\pm 1}]} (\pi_*(CP_b)) \cong \wedge^{\ell}_\bQ(H^{-*}(W)) \otimes_\bQ \bQ[t^{\pm 1}]$.

Under hypothesis (H1) the cohomology $H^*(W)$ is concentrated in even degrees, so if it has total degree $k$ then we have
$$\wedge^{k+1}_{\bQ[t^{\pm 1}]} (\pi_*(CP_b)) \cong \wedge^{k+1}_\bQ(H^{-*}(W)) \otimes_\bQ \bQ[t^{\pm 1}]=0$$
and so it follows from Lemma \ref{lem:SchurKills} that $\wedge ^{k+1} CP \simeq *$.  
As $\pi: E \to B$ is a fibre bundle with compact fibres Corollary \ref{cor:CCPDualisable} applies to it, so $CP$ is dualisable in $\HPBmod$ and hence the discussion of the previous section applies. Thus, as $CP$ is a ring object in $\HPBmod$, for any 
$$x \in H^{2p}(E) \subset \bigoplus_{i \in \bZ} H^{-2i}(E) = [HP_B\bQ, CP]_{\HPBmod}$$
multiplication by $x$ yields an endomorphism $\hat{x} : CP \to CP$, and in this case composing the map \eqref{eq:trace1} with $1 \in [HP_B\bQ, CP]_{\HPBmod}$ gives the identity
\begin{equation}\label{eq:trace2}
0 = \sum_{\sigma \in \Sigma_{k+1}} \mathrm{sign}(\sigma) \cdot \mathrm{Tr}(\hat{x}^{\circ l(\gamma_2)})\cdots \mathrm{Tr}(\hat{x}^{\circ l(\gamma_{q(\sigma)})}) \cdot {x}^{\circ l(\gamma_1)-1}
\end{equation}
in $[HP_B\bQ, CP]_{\HPBmod}$, because the partition $\lambda = (1^{k+1})$ corresponds to the sign representation.

\begin{corollary}\label{cor:EvRelations}
The polynomial
$$\rho_x(z) := \frac{(-1)^k}{k!}\sum_{\sigma \in \Sigma_{k+1}} \mathrm{sign}(\sigma) \cdot \trf^*_\pi({x}^{l(\gamma_2)})\cdots \trf^*_\pi(x^{l(\gamma_{q(\sigma)})}) \cdot z^{l(\gamma_1)-1} \in H^{2*}(B)[z]$$
is monic of degree $k$ and satisfies $\rho_x(x)=0 \in H^{2*}(E)$.
\end{corollary}
\begin{proof}
Let $E^\vee$ be a dual of $E \in \mathsf{Ho}(\mathsf{Sp}_{/B})$, so $CP^\vee := F_B(E^\vee, HP_B\bQ)$ is a dual in $\HPBmod$ of $CP$. The Becker--Gottlieb transfer $\trf^*_\pi$ is the map on cohomology induced by the composition
$$S^0_B \overset{\eta}\lra E \wedge_B E^\vee \cong E^\vee \wedge_B E \overset{E^\vee \wedge_B \Delta}\lra E^\vee \wedge_B  E \wedge_B E \overset{\epsilon \wedge_B E}\lra  E.$$
Applying $F_B(-, HP_B\bQ)$, this is
$$CP \overset{\eta \wedge CP}\lra CP^\vee \otimes CP \otimes CP \overset{CP^\vee \otimes \mu}\lra CP^\vee \otimes CP \cong CP \otimes CP^\vee \overset{\epsilon}\lra HP_B\bQ.$$

If $t \in \bigoplus_{i \in \bZ} H^{-2i}(E) = [HP_B\bQ, CP]_{\HPBmod}$ then composing the previous map with $t$ gives the class $\trf_\pi^*(t) \in \bigoplus_{i \in \bZ} H^{-2i}(B) = [HP_B\bQ, HP_B\bQ]_{\HPBmod}$. The commutative diagram
\begin{equation*}
\xymatrix{
HP_B\bQ \ar[r]^-{t} \ar[d]^-{\eta}& CP \ar[r]^-{\eta \wedge CP}& CP^\vee \otimes CP \otimes CP \ar[r]^-{CP^\vee \otimes \mu}& CP^\vee \otimes CP \ar[r]^\cong & CP \otimes CP^\vee\ar[d]^-{\epsilon}\\
CP^\vee \otimes CP \ar[rrru]_-{CP^\vee \otimes \hat{t}} & & & & HP_B\bQ
}
\end{equation*}
shows that $\trf_\pi^*(t)$ is the trace of $\hat{t}^\vee : CP^\vee \to CP^\vee$, which is the same as the trace of $\hat{t} : CP \to CP$: thus 
$$\trf^*_\pi(t) = \mathrm{Tr}(\hat{t}) \in [HP_B\bQ, HP_B\bQ]_{\HPBmod} = \bigoplus_{i \in \bZ} H^{-2i}(B).$$

In particular we have $\mathrm{Tr}(\hat{x}^{\circ i}) = \trf^*_\pi(x^i)$. Substituting this into \eqref{eq:trace2} therefore shows that
$$\sum_{\sigma \in \Sigma_{k+1}} \mathrm{sign}(\sigma) \cdot \trf^*_\pi({x}^{l(\gamma_2)})\cdots \trf^*_\pi(x^{l(\gamma_{q(\sigma)})}) \cdot x^{l(\gamma_1)-1}=0$$
in $H^{2pk}(E) \subset [HP_B\bQ, CP]_{\HPBmod}$. The coefficient of $x^{k}$ is the sum over the $k!$-many $(k+1)$-cycles $\sigma \in \Sigma_{k+1}$ of $\mathrm{sign}(\sigma) = (-1)^k$, which is $k! (-1)^k$. Thus after dividing by this coefficient we see that $\rho_x(x)=0$ as required.
\end{proof}

Applying this to the universal fibre bundle $p : B\Diff^+(W,*) \to B\Diff^+(W)$ and the cohomology class $s^*c$, and using the identity $\trf_p^*((s^*c)^i) = \kappa_{c^i e}(p)$ from \eqref{eq:TrfEuler}, one obtains a monic polynomial $\rho_c(z) \in R^*(W)[z]$ such that  $\rho_c(s^*c) =0 \in R^*(W,*)$ %is zero when evaluated on any smooth oriented fibre bundle with section over a finite CW-complex $B$. As a rational cohomology class is zero if and only if it vanishes when evaluated against every rational homology class, and a homology class may always be supported on a finite CW-complex, it follows that $\rho_c(s^*c)=0 \in R^*(W,*)$ 
and hence that $R^*(W,*)$ is integral over $R^*(W)$. Theorem \ref{mainThm:A} under hypothesis (H1) follows by applying Proposition \ref{prop:Integral}.

\subsection{Proof of Theorem \ref{mainThm:A} under the second hypothesis}

In this case we will work with the non-periodic chains $C$. We shall first prove the following generalisation of a theorem of Grigoriev \cite{grigoriev-relations}.

\begin{theorem}\label{thm:Grigoriev}
Let $W$ be a manifold of dimension $2n$ having rational cohomology only in degrees 0, $2n$, and odd degrees, and let $d := \dim_\bQ H^{odd}(W)$. Let $\pi : E \to B$ be a smooth oriented fibre bundle with fibre $W$. Let $a, b \in H^*(E)$ satisfy $\pi_!(a)=\pi_!(b)=0$, and $a$ have even degree. Then
$$\pi_!(a^2)^{\lceil \tfrac{d+1}{2}\rceil}=0 \quad \text{and}\quad \pi_!(ab)^{d+1}=0.$$
\end{theorem}

To begin with, we prove the following extension of Corollary \ref{cor:CCPDualisable}, which is the appropriate form of Poincar{\'e} duality in our setting.

\begin{lemma}\label{lem:selfdual}
If $\pi: E \to B$ is a Hurewicz fibration over a CW-complex, its fibre $F$ has the homotopy type of a finite Poincar{\'e} complex of dimension $n$, and $\pi_1(B)$ acts trivially on $H^n(F)$, then an orientation of $F$ determines an identification of the dual of $C$ with $\Sigma^n C$ in $\HBmod$.
\end{lemma}
\begin{proof}
This is proved in Section 3.1 of \cite{HLLRW} in dual form, where, passing to rational coefficients, the equivalence is expressed as $D^{fw}_E : \Sigma^n F_B(E, H_B \bQ) \overset{\sim}\to E \wedge_B H_B\bQ$. The domain of this morphism is $\Sigma^n C$ and as $C = F_B(E, H_B\bQ) = F_{\HBmod}(E \wedge_B H_B\bQ, H_B\bQ)$ we recognise $E \wedge_B H_B\bQ$ as the dual of $C$.
\end{proof}

%\begin{remark}
%See \cite[Section 3.1]{HLLRW} for an alternative proof of this lemma, using more parameterised stable homotopy theory, which avoids the assumption that $B$ is finite.
%\end{remark}

We now consider a smooth oriented fibre bundle $\pi : E \to B$ as in the statement of Theorem \ref{thm:Grigoriev}, with $B$ a CW-complex; this satisfies the hypotheses of the previous lemma. Fibre integration $\pi_! : H^*(E) \to H^{*-2n}(B)$ is realised in the category $\HBmod$ by the morphism $\pi_! : C \to \Sigma^{-2n} H_B\bQ$ dual to the unit $\iota : H_B\bQ \to C$, using the self-duality of $C$ described in Lemma \ref{lem:selfdual}. We may thus define an $H_B\bQ$-module $D'$ by the homotopy fibre sequence
$$D' \lra C \overset{\pi_!}\lra \Sigma^{-2n} H_B\bQ.$$
The composition $H_B\bQ \overset{\iota}\to C \overset{\pi_!}\to \Sigma^{-2n} H_B\bQ$ is null, as it represents the class
$$\pi_!(1) = 0 \in H^{-2n}(B) = [H_B\bQ, \Sigma^{-2n} H_B\bQ]_{\HBmod},$$
so $\iota$ lifts to a map $\iota' : H_B\bQ \to D'$ and we can define an $H_B\bQ$-module $D$ by the homotopy cofibre sequence
$$H_B\bQ \overset{\iota'}\lra D' \lra D.$$

\begin{lemma}
If $W$ only has rational cohomology in degree 0, $2n$, and odd degrees, and $d := \dim_\bQ H^{odd}(W)$ then $\mathrm{Sym}^{d+1}(D) \simeq *$.
\end{lemma}
\begin{proof}
The $H\bQ$-module spectrum $D_b$ is obtained by forming the homotopy fibre sequence $D'_b \to F(W, H\bQ) \overset{\pi_!}\to \Sigma^{-2n} H\bQ$ and then the homotopy cofibre sequence $H\bQ \overset{\iota'_b}\to D'_b \to D_b$.
Now $\pi_*(F(W, H\bQ)) = H^{-*}(W)$, and the map 
$$(\pi_!)_* :\pi_{*}(F(W, H\bQ)) = H^{-*}(W) \lra \pi_{*}(\Sigma^{-2n} H\bQ)$$
realises capping with the fundamental class so is surjective. By the associated long exact sequence we have
\begin{align*}
\pi_0(D'_b) &= \bQ\{1\}\\
\pi_{odd}(D'_b) &= H^{-odd}(W)
\end{align*}
and the remaining even homotopy groups are zero. Now the map
$$(\iota'_b)_* : \bQ = \pi_0(H\bQ) \lra \pi_0(D'_b)$$
realises the unit so is an isomorphism, and it follows that $\pi_{odd}(D_b) = H^{-odd}(W)$ and the even homotopy groups of $D_b$ vanish. As the $(d+1)$-st symmetric power of the graded $\bQ$-module $H^{-odd}(W)$ vanishes,  the rest follows from Lemma \ref{lem:SchurKills}.
\end{proof}

The map
$$D' \otimes D' \lra C \otimes C \overset{\mu}\lra C \overset{\pi_!}\lra \Sigma^{-2n} H_B\bQ$$
is null when precomposed with $D' \otimes H_B\bQ \overset{D' \otimes \iota'}\to D' \otimes D'$ or $H_B\bQ \otimes D' \overset{\iota' \otimes D'}\to D' \otimes D'$, so taking homotopy cofibres of these two maps gives a morphism
$$\phi : D \otimes D \lra \Sigma^{-2n} H_B\bQ.$$

\begin{proof}[Proof of Theorem \ref{thm:Grigoriev}]
If $a \in H^{-k}(E) = [\Sigma^k H_B\bQ, C]_{\HBmod}$ is such that $\pi_!(a)=0$ then it lifts to a map to $D'$ and hence determines a map $\bar{a}: \Sigma^k H_B\bQ \to D$. Similarly if $b \in H^{-\ell}(E)$ satisfies $\pi_!(b)=0$ then it gives a $\bar{b} : \Sigma^\ell H_B\bQ \to D$. The class $\pi_!(a \cdot b)$ may therefore be represented by
$$\Sigma^k H_B\bQ \otimes \Sigma^\ell H_B\bQ \overset{\bar{a} \otimes \bar{b}}\lra D \otimes D \overset{\phi}\lra \Sigma^{-2n} H_B\bQ.$$
Hence the class $\pi_!(a \cdot b)^N$ may be written as
$$(\Sigma^k H_B\bQ)^{\otimes N} \otimes (\Sigma^\ell H_B\bQ)^{\otimes N} \overset{\bar{a}^{N} \otimes \bar{b}^N}\lra D^{\otimes N} \otimes D^{\otimes N} \overset{\phi^N}\lra (\Sigma^{-2n} H_B\bQ)^{\otimes N},$$
as the degree $k$ of $\bar{a}$ is even so no sign is incurred in rearranging the factors. As $k$ is even, the map $\bar{a}^N : (\Sigma^k H_B\bQ)^{\otimes N} \to D^{\otimes N}$ factors through $\mathrm{Sym}^N(D)$ which is contractible as long as $N \geq d+1$. Hence $\pi_!(a \cdot b)^{d+1}=0$.

Similarly, if $a=b$ then we can choose $\bar{b} = \bar{a} : \Sigma^k H_B\bQ \to D$ in which case the map $\bar{a}^N \otimes \bar{a}^N : (\Sigma^k H_B\bQ)^{\otimes N} \otimes (\Sigma^k H_B\bQ)^{\otimes N} \to D^{\otimes N} \otimes D^{\otimes N}$ factors through $\mathrm{Sym}^{2N}(D)$, which is contractible as long as $2N \geq d+1$. Hence $\pi_!(a^2)^{\lceil \tfrac{d+1}{2}\rceil}=0$.
%This proves the theorem when $B$ is a finite CW-complex; it follows for general $B$ as a rational cohomology class vanishes if and only if it pulls back to zero under any map from a finite CW-complex.
\end{proof}

Now that we have Theorem \ref{thm:Grigoriev}, the entirety of Section 5 of \cite{grigoriev-relations} goes through with only notational changes, as this only uses the statement of Grigoriev's theorem. In particular, for $p \in H^{*}(BSO(2n))$ of even degree and $\chi = \chi(W) \neq 0$, the analogue of \cite[Example 5.19]{grigoriev-relations} gives the relation
\begin{equation*}
\left(p - \frac{\kappa_{ep}}{\chi} - \frac{e \kappa_p}{\chi} + \frac{\kappa_{e^2} \kappa_p}{\chi^2}\right)^{d+1}=0 \in R^*(W,*).
\end{equation*}
From this it is clear that $R^*(W,*)$ is a finite $R^*(W)$-module, as the monomials in $\bQ[p_1, p_2, \ldots, p_{n-1}, e]$ where no variable occurs with exponent larger than $d$ give a finite set of module generators. Thus $R^*(W,*)$ is integral over $R^*(W)$, so by Proposition \ref{prop:Integral} the algebra $R^*(W)$ is finitely-generated. This proves Theorem \ref{mainThm:A} under hypothesis (H2).

\subsection{Tautological relations}

Under either hypothesis we have established more than Theorem \ref{mainThm:A}, as we have produced explicit relations in $R^*(W,*)$. Under hypothesis (H2) these relations are equal to those obtained by Grigoriev, and under hypothesis (H1) they are given by Corollary \ref{cor:EvRelations} as
$$0 = \sum_{\sigma \in \Sigma_{k+1}} \mathrm{sign}(\sigma) \cdot \kappa_{e{c}^{l(\gamma_2)}}\cdots \kappa_{ec^{l(\gamma_{q(\sigma)})}} \cdot c^{l(\gamma_1)-1} \in R^*(W,*)$$
for each $c \in H^*(BSO(2n))$, where $k = \dim_\bQ H^*(W)$. These may of course be pushed forward to obtain relations in $R^*(W)$.

More generally, the trace identity technique of Section \ref{sec:TraceId} may be used to find relations among tautological classes for \emph{any} manifold. Recall that given a fibre bundle $W \to E \overset{\pi}\to B$ we have formed an associated object $CP \in \HPBmod$. Let us write $d_{ev} = \dim_\bQ H^{ev}(W)$ and $d_{odd} = \dim_\bQ H^{odd}(W)$. The first ingredient is the following consequence of a calculation of Deligne.

\begin{lemma}
If $\lambda$ is a partition whose Young diagram contains the rectangle $(d_{ev}+1) \times (d_{odd}+1)$, then $S_\lambda(CP) \simeq *$.
\end{lemma}
\begin{proof}
By Lemma \ref{lem:SchurKills} it is enough to verify that $S_{\lambda}(\pi_*(CP_b))=0$ for all $b \in B$. But we have shown that $\pi_*(CP_b) \cong H^{-*}(W)\otimes \bQ[t^{\pm 1}]$ as graded $\bQ[t^{\pm 1}]$-modules so $S_{\lambda}(\pi_*(CP_b))$ vanishes if $S_\lambda(H^{-*}(W))$ does, where the latter Schur functor is taken in $\mathsf{V}_\bQ$. By \cite[Corollary 1.9]{Deligne} a $(\bZ/2$-)graded vector space is annihilated by $S_\lambda(-)$ under the given assumption on its (super)dimension.
\end{proof}

In particular, for a given manifold $W$ we may take $\lambda$ to be the partition of $n = (d_{ev}+1) \cdot (d_{odd}+1)$ with Young diagram equal to the rectangle $(d_{ev}+1) \times (d_{odd}+1)$, so that we have $S_\lambda(CP) \simeq *$ and hence by \eqref{eq:trace1} we have the relation
\begin{equation*}
0 = \sum_{\sigma \in \Sigma_{n}} \chi_\lambda(\sigma) \cdot \kappa_{ec^{ l(\gamma_2)}}\cdots \kappa_{ec^{ l(\gamma_{q(\sigma)})}} \cdot c^{l(\gamma_1)-1} \in R^*(W,*).
\end{equation*}

It is a simple exercise with the Murnaghan--Nakayama rule to show that the character $\chi_\lambda$ vanishes on all $n$-cycles if both $d_{ev} > 0$ and $d_{odd}>0$. As $d_{ev}$ cannot be zero, because $H^0(W) \neq 0$, it follows that this relation is a monic polynomial in $c$ (after perhaps scaling by a rational number) if and only if $d_{odd}=0$. (This accounts for why we restricted to manifolds with only even rational cohomology in the first case of Theorem \ref{mainThm:A}.)

\section{Torus actions}

In this section we suppose that we have a smooth action of the torus $T = (S^1)^k$ on a $d$-dimensional orientable manifold $W$. We write $W^T$ for the fixed set of this action. The Borel construction gives a smooth fibre bundle
\begin{equation}\label{eq:Borel}
W \lra W \hcoker T \overset{\pi}\lra BT,
\end{equation}
and the action of $T$ on the tangent bundle $TW \to W$ gives a vector bundle $T^TW := TW \hcoker T \to W \hcoker T$, which is the vertical tangent bundle of the smooth fibre bundle $\pi$. Following the usual notation of equivariant cohomology we write 
$$H^*_T = H^*(BT;\bQ) = \bQ[x_1, x_2, \ldots, x_k] \quad\text{ and }\quad H^*_T(W) = H^*(W \hcoker T ;\bQ).$$ 
As \eqref{eq:Borel} is a smooth fibre bundle, there is a ring homomorphism $\rho: R^*(W) \to H^*_T$, and we denote by $R^*_T \leq H^*_T$ its image. Pulling back $\pi$ along itself gives a smooth fibre bundle over $W \hcoker T$ with canonical section, and so a ring homomorphism $\rho_* : R^*(W, *) \to H^*_T(W)$, and we denote by $R^*_T(*) \leq H^*_T(W)$ its image.

Our goal in this section is to describe conditions on the manifold $W$ and the action of $T$ on $W$ which allow us to estimate the Krull dimension of $R^*(W)$ as $\Kdim(R^*(W)) \geq k$. We will regularly use the following standard piece of commutative algebra: when one ring is integral over another they have the same Krull dimension, by the ``going up" and ``going down" theorems \cite[Ch.\ 5]{AM}. Our most general result is as follows.

\begin{theorem}\label{thm:main}
Let $T$ act smoothly and effectively on a connected closed orientable manifold $W$. Let $V_1, V_2, \ldots, V_p$ be an enumeration of the $T$-representations arising as normal spaces to points on $W^T$, and let $B_i$ denote the Euler characteristic of the subspace of $W^T$ consisting of those path components having normal representation $V_i$. 

If some $y_i \in \bQ[y_1, y_2, \ldots, y_p]$ is integral over the subring generated by
$$\sum_{i=1}^p B_i y_i^n, \quad n=1,2,3,\ldots,$$
then $H^*_T$ is integral over $R^*_T$. In particular $\Kdim(R^*(W)) \geq k$.
\end{theorem}

It is perhaps not clear when the hypothesis of this theorem is likely to hold. The following lemma, which we learnt from \cite{BCES}, gives a simple criterion.

\begin{lemma}
Suppose that we have discarded the $B_i$ which are zero, and that this is not all of them. If the remaining numbers $B_1, B_2, \ldots, B_p$ have all partial sums non-zero, then $\bQ[y_1, y_2, \ldots, y_p]$ is finite over the subring generated by
\begin{equation}\label{eq:1}
\sum_{i=1}^p B_i y_i^n, \quad n=1,2,3,\ldots,
\end{equation}
and so every $y_i$ is integral over this subring.
\end{lemma}
\begin{proof}
Write $B \leq \bQ[y_1, y_2, \ldots, y_p]$ for the subring generated by the $\sum_{i=1}^p B_i y_i^n$, and $B^+$ for the subset of positive-degree elements. 

\vspace{1ex}

\noindent\textbf{Claim.} If $\sqrt{(B^+)}=(y_1, y_2, \ldots, y_p)$ then $\bQ[y_1, \ldots, y_p]$ is finite over $B$.

\vspace{1ex}

Our proof of this claim follows the discussion at \cite{MO}. Under the assumption the quotient ring $\bQ[y_1, y_2, \ldots, y_p]/(B^+)$ has every $y_i$ nilpotent, so is a finite $\bQ$-module; let $z_1, z_2, \ldots, z_m \in \bQ[y_1, y_2, \ldots, y_p]$ be lifts of these finitely-many generators, which can be taken to be homogeneous as the ideal $(B^+)$ is homogeneous. We claim that these generate $\bQ[y_1, y_2, \ldots, y_p]$ as a $B$-module; let $M \subset \bQ[y_1, y_2, \ldots, y_p]$ be the $B$-submodule that they generate.

As the $z_i$ are homogeneous, and $B$ is generated by homogeneous elements, $M$ is a graded submodule of $\bQ[y_1, y_2, \ldots, y_p]$ with the monomial-length grading. Suppose $p \in \bQ[y_1, y_2, \ldots, y_p]$ is an element of minimal grading which does not lie in $M$. Then we may write
$$p = \sum_{i=1}^m U_i z_i + \sum V_j b_j$$
with $U_i \in \bQ$, $V_j \in \bQ[y_1, y_2, \ldots, y_p]$, and $b_j \in (B^+)$. But the $b_j$ have strictly positive degree, so the $V_j$ have strictly smaller degree than $p$ so must lie in $M$, and hence $p$ does too, which proves the claim.

\vspace{1ex}

In order to prove the lemma we must therefore show that $(y_1, y_2, \ldots, y_p)=(0,0,\ldots,0)$ is the only simultaneous solution to the equations $\sum_{i=1}^p B_i y_i^n=0$ for $n \in \bN$. If $(y_1, y_2, \ldots, y_p) \in \bQ^p$ is a solution, then grouping terms with $y_i = y_j$ together we obtain \emph{distinct} rational numbers $\bar{y}_i$ solving the equations
$$\sum_{i=1}^q \bar{B}_i \bar{y}_i^n=0$$
where each $\bar{B}_i$ is a partial sum of the $B_i$, and hence non-zero by assumption. But this means that the vector $(\bar{B}_1 \bar{y}_1, \ldots, \bar{B}_q \bar{y}_q)$ is in the kernel of the (transposed) Vandermonde matrix associated to $(\bar{y}_1, \ldots, \bar{y}_q)$, so as the $\bar{y}_i$ are all distinct it follows that $(\bar{B}_1 \bar{y}_1, \ldots, \bar{B}_q \bar{y}_q)=0$, and as the $\bar{B}_i$ are all non-zero it follows that $\bar{y}_i=0$ as required.
\end{proof}

The following corollary, whilst not so powerful as Theorem \ref{thm:main}, is often easier to apply as one does not need to classify the normal representations at the fixed set.

\begin{corollary}\label{cor:main}
Let the path components $X_1, X_2, \ldots, X_\ell$ of the fixed set $W^T$ have Euler characteristics $A_1, A_2, \ldots, A_\ell$. If some $x_i \in \bQ[x_1, x_2, \ldots, x_\ell]$ is integral over the subring generated by
$$\sum_{i=1}^\ell A_i x_i^n, \quad n=1,2,3,\ldots,$$
then $H^*_T$ is integral over $R^*_T$. In particular $\Kdim(R^*(W)) \geq k$.
\end{corollary}
\begin{proof}
Consider the ring homomorphism $\phi : \bQ[x_1, x_2, \ldots, x_\ell] \to \bQ[y_1, y_2, \ldots, y_p]$ defined by sending $x_i$ to $y_j$ if the normal representation at every point of $X_i$ is $V_j$. Then
$$\phi\left(\sum_{i=1}^\ell A_i x_i^n\right) = \sum_{i=1}^\ell A_i \phi(x_i)^n = \sum_{j=1}^p B_j y_j^n$$
so $\phi$ sends the subring $A \subset \bQ[x_1, x_2, \ldots, x_\ell]$ generated by the $\sum_{i=1}^\ell A_i x_i^n$ onto the subring $B \subset \bQ[y_1, y_2, \ldots, y_p]$ generated by the $\sum_{i=1}^p B_i y_i^n$.

If $x_i$ is integral over $A$ then there is a polynomial $q(x) = \sum a_i x^i$ with coefficients in $A$ such that $q(x_i)=0$. Then $q'(y) = \sum \phi(a_i) y^i$ is a polynomial over $B$ such that $q'(\phi(x_i))=0$, so $y_j = \phi(x_i)$ is integral over $B$, and hence Theorem \ref{thm:main} applies.
\end{proof}

\begin{example}\label{ex:ConnFixSet}
There are several standard conditions which oblige a torus action on a manifold $W$ to have connected fixed-set. For example
\begin{enumerate}[(i)]
\item Let $W$ have dimension $2n$, and suppose that all its cohomology apart from $H^0(W;\bQ)$ and $H^{2n}(W;\bQ)$ lies in odd degrees, and that there is some cohomology in odd degrees. Then $W^T$ is connected (by the localisation theorem in equivariant cohomology, which we will describe in the following section).

\item If $W$ has trivial even-dimensional rational homotopy groups, then $W^T$ is empty or connected \cite[Theorem IV.5]{Hsiang}.
\end{enumerate}
In such cases $\chi(W^T) = \chi(W)$, so if this is non-zero then the hypotheses of Corollary \ref{cor:main} are satisfied.
\end{example}

\begin{example}
Suppose that the action of $T^k$ on $W$ has isolated fixed points, or more generally that all $A_i$ are equal and non-zero. Then the subring generated by the $\sum_{i=1}^\ell A_i x_i^n$ is the subring of symmetric polynomials in $\bQ[x_1, x_2, \ldots, x_\ell]$, and every $x_i$ is integral over this so the hypotheses of Corollary \ref{cor:main} are satisfied.

This immediately implies that if $W^{2n}$ is a quasitoric manifold (that is, the ``toric manifolds" of \cite{DJ}) then $\mathrm{Kdim}(R^*(W)) \geq n$, as such manifolds by definition have an action of an $n$-torus with isolated fixed points. Slightly more subtly, if $G/K$ is a homogeneous space of rank zero (i.e.\ $\mathrm{rk}(G) = \mathrm{rk}(K)$) then a common maximal torus $T$ of $G$ and $K$ acts on $G/K$ with fixed points given by the finite set $(W_G(T) \cdot K)/K \subset G/K$, where $W_G(T) := N_G(T)/T$ denotes the (finite) Weyl group of $G$, so $\mathrm{Kdim}(R^*(G/K)) \geq \mathrm{rk}(G)$.
\end{example}

\subsection{The localisation theorem}\label{sec:Loc}

We now prepare for the proof of Theorem \ref{thm:main}. Let $X_1, X_2, \ldots, X_\ell$ be the components of the fixed set $W^T$, with $d_i := \dim(X_i)$, let $\nu_{X_i}$ be the normal bundle of $X_i$ in $W$, and let $\nu_i$ be the $T$-representation which arises as each fibre of $\nu_{X_i}$. Let us write $A_i := \chi(X_i)$, and write $V_1, V_2, \ldots, V_p$ for an enumeration of the $T$-representations $\nu_i$ which arise. Then we have $B_i = \sum_{j\ s.t.\ \nu_j = V_i} A_j$.

Let us write $\rho_i : H^*_T(W) \to H^*_T(X_i)$ for the restriction map in equivariant cohomology, and $\pi_! : H^*_T(W) \to H^{*-d}_T$ and $(\pi_i)_! : H^*_T(X_i) \to H^{*-d_i}_T$ for the fibre integration maps. As the $T$-action on $X_i$ is trivial we have $X_i \hcoker T = BT \times X_i$, and so the fibre integration map $(\pi_i)_!$ is simply given by slant product with the fundamental class of $X_i$. As $T$ acts on the normal bundle $\nu_{X_i} \to X_i$, there is an induced vector bundle $\nu^T_{X_i} := \nu_{X_i} \hcoker T \to X_i \hcoker T$. 

Let $S \subset H^*_T$ be the multiplicative subset of nonzero elements. The localisation theorem in equivariant cohomology (of Borel \cite[XII.\S3]{Borel}, Hsiang \cite{Hsiang70} and Quillen \cite[Section 4]{Quillen}) says that the map
$$\bigoplus_i \rho_i : S^{-1}H^*_T(W) \lra \bigoplus\limits_{i} S^{-1}H^{*}_T(X_i)$$
is an isomorphism. Even more is true: Atiyah and Bott have shown \cite[eq (3.8)]{AB} that the class $e(\nu_{X_i}^T) \in S^{-1}H_T^{d-d_i}(X_i)$ is a unit and that we have a commutative diagram
\begin{equation}\label{eq:LocGysin}
\begin{gathered}
\xymatrix{
S^{-1}H^*_T(W) \ar[r]_-\sim^-{\bigoplus\limits_i \rho_i} \ar[d]^-{\pi_!} & {\bigoplus\limits_{i} S^{-1}H^{*}_T(X_i)} \ar[rr]^-{\bigoplus\limits_i e(\nu_{X_i}^T)^{-1}}_-{\sim} & & {\bigoplus\limits_{i} S^{-1}H^{*+d_i-d}_T(X_i)} \ar[d]^-{\sum_i (\pi_i)_!}\\
S^{-1}H^{*-d}_T \ar@{=}[rrr] & & & S^{-1}H^{*-d}_T.
}
\end{gathered}
\end{equation}
See \cite[p.\ 366]{AP} for a textbook exposition of the localisation theorem.

\subsection{Proof of Theorem \ref{thm:main}}

Using the diagram \eqref{eq:LocGysin} to compute 
$$\kappa_{ep_I} = \pi_!(e(T^TW) p_I(T^TW)) \in S^{-1}H^*_T,$$
which we know lies in the subring $H^*_T$, gives
\begin{align*}
\kappa_{ep_I} &= \sum_{i=1}^\ell (\pi_i)_!\left(\frac{e(TX_i \oplus \nu_{X_i}^T) p_I(TX_i \oplus \nu_{X_i}^T)}{e(\nu_{X_i}^T)}\right) \\
&= \sum_{i=1}^\ell (\pi_i)_!(e(TX_i) p_I(TX_i \oplus \nu_{X_i}^T))
\end{align*}
and in $H^*_T(X_i) = H^*_T \otimes H^*(X_i)$ we have 
$$p_I(TX_i \oplus \nu_{X_i}^T) = p_I(\nu_i) \otimes 1 + \text{terms with a nontrivial $H^*(X_i)$ component}.$$
When we multiply by $e(TX_i)$ and integrate over $X_i$ the latter terms do not contribute, so as $\int_{X_i} e(TX_i) = \chi(X_i) = A_i$ we get 
\begin{equation*}
\kappa_{ep_I} = \sum_{i=1}^\ell A_i p_I(\nu_i) \in H^*_{T}.
\end{equation*}
Grouping these terms by the representation types $V_j$ instead gives
\begin{equation}\label{eq:2}
\kappa_{ep_I} = \sum_{i=1}^p B_i p_I(V_i) \in H^*_{T}.
\end{equation}

Applying this to $p_I = p_j^n$ we find that 
$$\sum_{i=1}^p B_i p_j(V_i)^n \in R^*_T \text{ for all $j$ and $n$}.$$
Applying the hypothesis of the theorem for each $j$, we find that there exists an $i$ such that all $p_j(V_i)$ lie in a common integral extension $R^*_T \subseteq R' \subseteq H^*_T$. On the one hand $R'$ is integral over $R^*_T$. On the other hand by a theorem of Venkov \cite{venkov} the ring $H^*_T$ is finite over the subring generated by the $p_j(V_i)$ (because $V_i$ is a faithful representation of $T$, by the standard lemma given below), and hence is finite (and so integral) over $R'$. It follows that $H^*_T$ is integral over $R^*_T$, so in particular they have the same Krull dimension, namely $k$.

Finally, $R^*(W) \to R^*_T$ is surjective and so $\Kdim(R^*(W)) \geq \Kdim(R^*_T) = k$.

\begin{lemma}
If $T$ acts effectively and smoothly on a connected closed manifold $W$, then any $T$-representation arising as the normal space to a point on $W^T$ is faithful.
\end{lemma}
\begin{proof}
We may choose a $T$-invariant Riemannian metric on $W$, so the exponential map $\exp : T W \to W$ is equivariant; the restriction of the exponential map to a fibre $T_x W \to W$ is a diffeomorphism when restricted to a neighbourhood of $0 \in T_x W$. 

If the action of $T$ on the normal space $V$ to $W^T$ at $x$ had a non-trivial kernel $\{e\} < T' \leq T$ then the $T'$-action on $T_x W = T(W^T) \oplus V$ is trivial. By exponentiating, it follows that $T'$ fixes an open neighbourhood of $x \in W$. Thus the fixed set $W^{T'}$ is a submanifold of $W$ which contains an open subset; as $W$ is connected it follows that it is the whole of $W$. This contradicts the action being effective.
\end{proof}

\subsection{An extension}

The discussion so far gives a technique more general Theorem \ref{thm:main}, but difficult to formalise in a single result. It is best described through an example.

\begin{proposition}\label{prop:loc}
Let $T$ act effectively on $W$ with two fixed components $X_1$ and $X_2$. Suppose that $\chi(X_1) = -\chi(X_2) \neq 0$ but that the normal $T$-representations $\nu_1$ and $\nu_2$ at $X_1$ and $X_2$ have all Pontrjagin classes distinct (when they are non-zero). Then $\Kdim(R^*(W)) \geq k$.
\end{proposition}
\begin{proof}
We have that
$$\tfrac{1}{\chi(X_1)}\kappa_{e p_j^n} = p_j(\nu_1)^n-p_j(\nu_2)^n \in R^*_T \leq H^*_T$$
for all $j$ and $n$, and $p_j(\nu_1)-p_j(\nu_2) \neq 0 \in R^*_T$. Hence
$$p_j(\nu_1) = \frac{1}{2} \left( p_j(\nu_1)-p_j(\nu_2) + \frac{p_j(\nu_1)^2-p_j(\nu_2)^2}{p_j(\nu_1)-p_j(\nu_2)}\right) \in R^*_T[(p_j(\nu_1)-p_j(\nu_2))^{-1}].$$
Therefore after inverting the finite set
$$S := \{p_j(\nu_1)-p_j(\nu_2), j = 1,2,\ldots\}$$
of non-zero elements in $R^*_T \leq H^*_T = \bQ[x_1, x_2, \ldots, x_k]$, we find that the $p_j(\nu_1)$ lie in $S^{-1}R^*_T$, and hence by Venkov's theorem \cite{venkov} that $S^{-1} H^*_T$ is a finite $S^{-1}R^*_T$-module. As $S^{-1} H^*_T$ still has Krull dimension $k$ (there is a maximal ideal $\mathfrak{m}$ of $H_T^*$ not containing the product of the finitely-many elements in $S$---as the intersection of all maximal ideals is zero---whence $(S^{-1} H^*_T)_{S^{-1}\mathfrak{m}} \cong (H^*_T)_{\mathfrak{m}}$ so $S^{-1}\mathfrak{m}$ is a maximal ideal of $S^{-1} H^*_T$ of height $k$), it follows that $S^{-1}R^*_T$ has Krull dimension $k$ and so $\Kdim(R^*_T) \geq k$. Hence $\Kdim(R^*(W)) \geq k$.
\end{proof}

\section{Examples}\label{sec:Examples}

\subsection{Manifolds with mostly odd cohomology}\label{sec:OddCohom}

Let $W$ be a $2n$-dimensional manifold whose cohomology is only non-trivial in degrees 0, $2n$, and odd degrees, let $d = \dim_\bQ H^{odd}(W)$, and suppose $\chi(W) = 2-d \neq 0$. Then by Theorem \ref{mainThm:A} the $\bQ$-algebra $R^*(W)$ is finitely-generated and $R^*(W,*)$ is a finite $R^*(W)$-module.

Furthermore, by our method of proof, Grigoriev's theorem holds for these manifolds (our Theorem \ref{thm:Grigoriev}). Therefore the results of Sections 2 and 3 of \cite{galagriran-characteristic} hold for $W$ as well, as Grigoriev's theorem was the only external input. So if $d > 2$ then
$$\bQ[\kappa_{ep_1}, \ldots, \kappa_{ep_{n-1}}] \lra R^*(W)/\sqrt{0}$$
is surjective. Hence $\Kdim(R^*(W)) \leq n-1$.

By Example \ref{ex:ConnFixSet} (i), if $T = (S^1)^k$ acts on such a manifold $W$ then the fixed set $W^T$ is connected, so $\Kdim(R^*(W)) \geq k$. The construction of \cite[Section 4.1]{galagriran-characteristic} can be mimicked to obtain an action of $SO(k) \times SO(2n-k)$ on $\#^g S^k \times S^{2n-k}$ for any $k$, and the calculation of the characteristic classes $\kappa_{ep_i}$ for the associated bundle is entirely analogous. 

We obtain the following generalisation of the results of \cite{galagriran-characteristic}.

\begin{corollary}\label{cor:GGRGeneralisation}
For $k$ odd and $g > 1$ we have
$$\bQ[\kappa_{ep_1}, \ldots, \kappa_{ep_{n-1}}] \overset{\sim}\lra R^*(\#^g S^k \times S^{2n-k})/\sqrt{0}$$
and
$$R^*(\#^g S^k \times S^{2n-k})/\sqrt{0} \overset{\sim}\lra R^*(\#^g S^k \times S^{2n-k}, *)/\sqrt{0}.$$
Furthermore $(2-2g) \cdot c = \kappa_{ec} \in R^*(\#^g S^k \times S^{2n-k}, *)/\sqrt{0}$, so 
$$R^*(\#^g S^k \times S^{2n-k}, D^{2n})/\sqrt{0} = \bQ,$$
and hence $R^*(\#^g S^k \times S^{2n-k}, D^{2n})$ is a finite-dimensional $\bQ$-vector space.
\end{corollary}

As in \cite{galagriran-characteristic} results can be obtained for $g=0$ or $1$ too, but we shall not write them out here.

\subsection{Quasitoric manifolds}

A quasitoric manifold $W^{2n}$ has by definition a smooth action of $T = (S^1)^n$ with isolated fixed points, so has $\Kdim(R^*(W)) \geq n$ by Corollary \ref{cor:main} . Furthermore, the integral cohomology of $W$ is supported in even degrees, so its rational cohomology is too, and therefore by Theorem \ref{mainThm:A} the $\bQ$-algebra $R^*(W)$ is finitely-generated and $R^*(W,*)$ is a finite $R^*(W)$-module.

\subsection{Non-finite generation}

We shall give some examples of manifolds $W$ for which $R^*(W)$, and in fact even $R^*(W)/\sqrt{0}$, is not finitely-generated. We shall do so by constructing actions of a torus $T$ on $W$ and showing that the tautological subring $R^*_T \leq H^*_T$ is not finitely-generated. As $H^*_T$ is an integral domain the natural surjection $R^*(W) \to R^*_T$ factors through $R^*(W)/\sqrt{0}$, which therefore shows that $R^*(W)/\sqrt{0}$ is not finitely-generated. 

Before attempting this method there is an important observation to be made.

\begin{obs}
Let $T = (S^1)^k$ act on $W$ satisfying the hypotheses of Theorem \ref{thm:main}; then that theorem shows that the inclusion $R^*_T \hookrightarrow H^*_T$ is integral. 

As $H^*_T$ is Noetherian, and $H^*(BT ; H^*(W))$ is a finitely-generated 
%free 
$H^*_T$-module, it follows from the Serre spectral sequence for the Borel construction that $H^*_T(W)$ is a finitely-generated $H^*_T$-module and hence is integral over $H^*_T$. 

Therefore the morphism $R^*_T \to H^*_T \to H^*_T(W)$ is integral, so $R^*_T \to R^*_T(*)$ is integral too. It then follows from applying Lemma \ref{lem:AlgIntegral} as in the proof of Proposition \ref{prop:Integral} that $R^*_T \to R^*_T(*)$ is finite and $R^*_T$ is a finitely-generated $\bQ$-algebra.
\end{obs}

So to pursue the programme we have suggested one should only try to use torus actions which \emph{do not} satisfy the hypotheses of Theorem \ref{thm:main}. The following allows us to construct manifolds with torus actions having prescribed normal representations and Euler characteristics of its fixed sets.

\begin{construction}\label{const:1}
Fix a positive odd integer $n$ and an even integer $k$. Let $\Sigma(k)^{2n}$ be the $2n$-manifold of Euler characteristic $k$ obtained as $\#^g S^n \times S^n$ (if $k$ is non-positive) or $\coprod^g S^{2n}$ (if $k$ is positive). Let $H(k)^{2n+1}$ be the manifold with boundary $\Sigma(k)^{2n}$ given by $\natural^g S^n \times D^{n+1}$ or $\coprod^g D^{2n+1}$ respectively.

Let $T$ be a torus, and suppose we are given even integers $B_1, B_2, \ldots, B_p$ and distinct faithful complex $T$-representations $V_1, V_2, \ldots, V_p$, which are all of the same dimension and which have no trivial subrepresentations. Then we can form the manifold
$$M(i) = M(B_i, V_i) := H(B_i)^{2n+1} \times \bS(V_i) \cup_{\Sigma(B_i) \times \bS(V_i)} \Sigma(B_i)^{2n} \times \bD(V_i).$$
which has a $T$-action on the right-hand factors. We may then let $M$ be the disjoint union $M = M(1) \sqcup M(2) \sqcup \cdots \sqcup M(p)$.
\end{construction}

As $V_i$ is  representation having no trivial subrepresentations, $T$ acts freely on $\bS(V_i)$ and its only fixed point on $\bD(V_i)$ is 0. Thus $M(i)^T = \Sigma(B_i)^{2n} \times \{0\}$, and the normal representation at these fixed points is given by $V_i$.

Each $V_i$ may be written as a sum $L_1 \oplus \cdots \oplus L_m$ of 1-dimensional complex $T$-representations; if a unit vector $v \in \bS(V_i)$ is written in components as $(l_1, \ldots, l_m)$ with all $l_j$ non-zero, then a $t \in T$ which stabilises it must act trivially on each $L_j$, so must act trivially on $V_i$, so $t$ must be the identity as $V_i$ is a faithful $T$-representation. Thus such a $v \in \bS(V_i)$ must lie in a free orbit, so in particular each path component of $M(i)$ has a free orbit. If one prefers a connected manifold, such free orbits in two different path components have tubular neighbourhoods $T$-equivariantly diffeomorphic to $T \times D^{2n+2m -\mathrm{rk}(T)}$, which can therefore be cut out and the remaining pieces glued together $T$-equivariantly along the common boundaries $T \times S^{2n+2m -\mathrm{rk}(T)-1}$. Doing this enough times yields a connected $T$-manifold with the same fixed-point data, and hence by localisation with the same characteristic classes.

\begin{lemma}
The $T$-manifold $M$ so obtained has $\kappa_{p_I}=0$ and
$$\kappa_{e p_I} = \sum_{i=1}^p B_i \cdot p_I(V_i) \in H^*_T.$$
\end{lemma}
\begin{proof}
The second statement follows from \eqref{eq:2}. An analogous calculation shows that
$$\kappa_{p_I} = \sum_{i=1}^p (\pi_i)_! \left(\frac{p_I(TX_i \oplus \nu_{X_i}^T)}{e(\nu_{X_i}^T)}\right).$$
The bundle $\nu_{X_i} \to X_i$ is trivial, so the equivariant bundle $\nu_{X_i}^T$ is isomorphic to the pullback of $V_i$ to $X_i \hcoker T = BT \times X_i$. Thus the total Pontrjagin class satisfies
$$p(TX_i \oplus \nu_{X_i}^T) = p(V_i) \otimes p(TX_i) = p(V_i) \otimes 1 \in H^*_T \otimes H^*(X_i)$$
as $TX_i$ is stably trivial, and so $p_j(TX_i \oplus \nu_{X_i}^T) = p_j(V_i) \otimes 1$. Hence 
$$\frac{p_I(TX_i \oplus \nu_{X_i}^T)}{e(\nu_{X_i}^T)} = \frac{p_I(V_i)}{e(V_i)} \otimes 1$$
which pushes forward to zero (as $\dim (X_i) = 2n >0$), so $\kappa_{p_I}=0$.
\end{proof}

We now give our example.

\begin{example}
Let $T = (S^1)^2$ and $V_1$ be the 2-dimensional complex $T$-representation with weights $\{x_1+x_2, x_2\}$, and $V_2$ be the 2-dimensional complex $T$-representation with weights $\{x_1, x_2\}$. Construction \ref{const:1} with $B_1=2$ and $B_2=-2$ yields a $T$-manifold $W$ (which may be chosen to have any dimension at least 6 and congruent to 2 modulo 4) having $\kappa_{p_I}=0$ and
$$\kappa_{ep_I} = 2(p_I(V_1) - p_I(V_2)) \in H^*_T = \bQ[x_1, x_2].$$
For the chosen representations the total Pontrjagin classes are
\begin{align*}
p(V_1) &= (1-(x_1+x_2)^2)(1-x_2^2)\\
p(V_2) &= (1-x_1^2)(1-x_2^2).
\end{align*}
Let us consider the image of the tautological subring $R^*_T \leq H^*_T = \bQ[x_1, x_2]$ in the quotient $\bQ[x_1, x_2]/(x_2^2)$. Here $p_2(V_1) = p_2(V_2)=0$ and
\begin{align*}
p_1(V_1) &= -(2x_1x_2 + x_1^2)\\
p_1(V_2) &= -x_1^2,
\end{align*}
so the only non-zero $\kappa_{ep_I}$ in this quotient ring are
$$\kappa_{ep_1^i} = 2(-1)^i((2x_1x_2 + x_1^2)^i - (x_1^2)^i) = 4i(-1)^ix_1^{2i-1} x_2,$$
so the image of $R^*_T$ in $\bQ[x_1, x_2]/(x_2^2)$ is the subring $S:=\bQ\langle x_1 x_2, x_1^{3} x_2, x_1^{5} x_2, \ldots \rangle$. The ring $S$ is an infinite-dimensional $\bQ$-vector space, as the $x_1^{2i-1} x_2$ all have different degrees and are non-zero as they are not divisible by $x_2^2$. On the other hand, multiplication of any two positive-degree elements in $S$ is zero, as each positive-degree element is divisible by $x_2$ so a product is divisible by $x_2^2$. Thus $S$ is infinitely-generated, so $R^*_T$ is too, and hence $R^*(W)/\sqrt{0}$ is too.
\end{example}

Let us record some observations about this example.

\begin{remark}
If we suppose that $n \geq 5$ is odd and the $T$-manifolds $M(2, V_1)$ and $M(-2, V_2)$ are glued along a free orbit as suggested above, then the $(2n+4)$-manifold $M$ obtained is simply-connected and has the same integral homology as 
$$(S^2 \times S^{2n+2}) \# (S^2 \times S^{2n+2}) \# (S^3 \times S^{2n+1}) \# (S^n \times S^{n+4}) \# (S^n \times S^{n+4}).$$
\end{remark}

\begin{remark}
Although this tautological ring is not finitely-generated, Proposition \ref{prop:loc} applies to this torus action and gives $\Kdim(R^*(W)) \geq 2$. (Specifically, we have
$$p_1(V_1) - p_1(V_2) = -(2x_1x_2 + x_2^2) \quad\quad\quad p_2(V_1) - p_2(V_2) = x_2^2(2x_1x_2 + x_2^2)$$
so after inverting $s := x_2(2x_1 + x_2) \neq 0 \in R^*_T$ the subring $s^{-1}R^*_T \leq s^{-1}H^*_T$ contains $p_1(V_1)$, $p_2(V_1)$, $p_1(V_2)$, and $p_2(V_2)$.)
\end{remark}

\begin{remark}
Choosing $* \in X_2$ gives a map $R^*(W,*) \to H^*_T$, whose image is generated by the $\kappa_{ep_I} = 2(p_I(V_1)-p_I(V_2))$ along with the characteristic classes of the representation $V_2$, which are $e(V_2) = x_1x_2$, $p_1(V_2)=-(x_1^2+x_2^2)$, and $p_2(V_2)=x_1^2 x_2^2$. Rearranging a little shows that this is the subring generated by $e(V_2)$ and the $p_j(V_i)$, so is finitely generated. (Similarly if we choose $* \in X_1$.) This raises the interesting possibility that $R^*(W, *)$ might be finitely-generated in more generality than $R^*(W)$ is.
\end{remark}

\subsection{The complex projective plane}\label{sec:exCP2}

Let us consider the manifold $\bC\bP^2$, whose cohomology is supported in even degrees. Thus by Theorem \ref{mainThm:A} the $\bQ$-algebra $R^*(\bC\bP^2)$ is finitely-generated and $R^*(\bC\bP^2, *)$ is a finite $R^*(\bC\bP^2)$-module. We will explain estimates on the generators for these algebras, using the relations developed in Section \ref{sec:TraceId}. The computations were done with assistance from {\tt Maple}\texttrademark.

The trace identity technique of Section \ref{sec:TraceId} gives the relation
\begin{equation*}
c^3 = \kappa_{ec} c^2 - \frac{\kappa_{ec}^2-\kappa_{ec^2}}{2!}c + \frac{\kappa_{ec}^3 -3 \kappa_{ec} \kappa_{ec^2} +2 \kappa_{ec^3}}{3!} \in R^*(\bC\bP^2, *)
\end{equation*}
for any $c \in H^*(BSO(4)) = \bQ[p_1, e]$. In particular, for $c=e$ and $c=p_1$ we obtain
\begin{align}
e^3 = \kappa_{e^2} e^2 - \frac{\kappa_{e^2}^2-\kappa_{e^3}}{2!}e + \frac{\kappa_{e^2}^3 -3 \kappa_{e^2} \kappa_{e^3} +2 \kappa_{e^4}}{3!}\label{eq:CP2rele}\\
p_1^3 = \kappa_{ep_1} p_1^2 - \frac{\kappa_{ep_1}^2-\kappa_{ep_1^2}}{2!}p_1 + \frac{\kappa_{ep_1}^3 -3 \kappa_{ep_1} \kappa_{ep_1^2} +2 \kappa_{ep_1^3}}{3!}\label{eq:CP2relp1}
\end{align}
We may partially polarise the relation by taking $c = e + t \cdot p_1$, expanding and collecting coefficients of powers of $t$. The coefficients of $1$ and of $t^3$ simply give the relations \eqref{eq:CP2rele} and \eqref{eq:CP2relp1}. The coefficient of $t$ gives
\begin{align}\label{eq:CP2rel2}
\begin{split}
&-\kappa_{e^3 p_1}-(1/2)\kappa_{e^2}^2\kappa_{ep_1}-\kappa_{e^2p_1}e-\kappa_{ep_1}e^2+(1/2)\kappa_{e^2}^2 p_1\\
&+\kappa_{e^2} \kappa_{e^2p_1}-(1/2)\kappa_{e^3} p_1 +(1/2)\kappa_{ep_1}\kappa_{e^3}+\kappa_{e^2}\kappa_{ep_1}e-2\kappa_{e^2}e p_1+3 e^2 p_1=0
\end{split}
\end{align}
and the coefficient of $t^2$ gives
\begin{align}\label{eq:CP2rel3}
\begin{split}
&(1/2)e\kappa_{ep_1}^2-(1/2)\kappa_{ep_1}^2\kappa_{e^2}-\kappa_{e^2p_1^2}-(1/2)e\kappa_{ep_1^2}+\kappa_{ep_1}\kappa_{e^2p_1}\\
&+(1/2)\kappa_{ep_1^2}\kappa_{e^2}-p_1\kappa_{e^2p_1}-p_1^2\kappa_{e^2}+p_1\kappa_{ep_1}\kappa_{e^2}-2e p_1\kappa_{ep_1}+3 e p_1^2=0.
\end{split}
\end{align}
(More generally, one could fully polarise this relation, by writing $c = u + t \cdot v + s \cdot w$, expanding out and taking the coefficient of $ts$: this gives a trilinear form in the variables $(u,v,w)$ which vanishes for all $u,v,w \in H^*(BSO(4)) = \bQ[p_1, e]$, and there is no reason to take these to be linear terms. However, we will not pursue this here.)

The relations \eqref{eq:CP2rele}, \eqref{eq:CP2relp1}, \eqref{eq:CP2rel2} and \eqref{eq:CP2rel3}, multiplied by monomials in $\bQ[p_1, e]$ and pushed forward, show that certain $\kappa_{e^a p_1^b} \in R^*(\bC\bP^2)$ are decomposable. Specifically
\begin{align*}
\kappa_{xp_1^3} &\text{ is decomposable for any monomial $x \neq 1, e, p_1$}\\
\kappa_{xe p_1^2} &\text{ is decomposable for any monomial $x \neq 1, e, p_1$}\\
\kappa_{xe^2 p_1} &\text{ is decomposable for any monomial $x \neq 1, e, p_1$}\\
\kappa_{xe^3} &\text{ is decomposable for any monomial $x \neq 1, e, p_1$.}
\end{align*}
Writing $\equiv$ to mean ``equal modulo decomposables", there are further relations:

\begin{enumerate}[(i)]
\item Pushing \eqref{eq:CP2relp1} forward gives $\kappa_{p_1^3} \equiv \frac{3}{2} \kappa_{ep_1^2}$.

\item Pushing \eqref{eq:CP2relp1} multiplied by $p_1$ forward gives $\kappa_{p_1^4} \equiv \kappa_{ep_1^3}$.

\item Pushing \eqref{eq:CP2rele} forward gives that $\kappa_{e^3}$ is decomposable, and in fact that $\kappa_{e^3} = \kappa_{e^2}^2$.

\item Pushing \eqref{eq:CP2rele} multiplied by $p_1$ forward gives $\kappa_{e^3p_1} \equiv \kappa_{e^4}$.

\item Pushing \eqref{eq:CP2rel2} multiplied by $p_1$ forward gives $\kappa_{e^2p_1^2} \equiv \kappa_{e^3 p_1}$.

\item Pushing \eqref{eq:CP2rel3} forward gives $2\kappa_{e^2p_1} \equiv  \kappa_{ep_1^2}$.

\item Pushing \eqref{eq:CP2rel3} multiplied by $p_1$ forward gives $\kappa_{ep_1^3} \equiv \kappa_{e^2p_1^2}$.
\end{enumerate}
Using these relations we find that the five classes
$$\kappa_{p_1^2}, \kappa_{p_1^3}, \kappa_{p_1^4}, \kappa_{ep_1}, \kappa_{e^2} \in R^*(\bC\bP^2)$$
generate. 

As described in \cite[Section 2]{galagriran-characteristic}, it follows from work of Atiyah \cite{atiyah-signature} that for each Hirzebruch class $\mathcal{L}_i$ the associated class $\kappa_{\mathcal{L}_i} \in R^*(W)$ is pulled back via the natural map
$$\phi : B\Diff^+(W) \lra B\mathrm{Aut}(H, \lambda),$$
where $H=H^n(W;\bZ)/\text{torsion}$ and $\lambda : H \otimes H \to \bZ$ is the intersection form of $W$. 

For $W=\bC\bP^2$ the bilinear form $(H, \lambda) = (\bZ, (1))$ has automorphism group $\bZ/2$, which has trivial rational cohomology. Thus the classes $\kappa_{\mathcal{L}_i} \in R^*(\bC\bP^2)$ are zero. The first few are
\begin{align*}
7\kappa_{e^2}-\kappa_{p_1^2}&=0\\
-13\kappa_{e^2p_1}+2\kappa_{p_1^3}&=0\\
-19\kappa_{e^4}+22\kappa_{e^2p_1^2}-3\kappa_{p_1^4}&=0\\
127\kappa_{e^4p_1}-83\kappa_{e^2p_1^3}+10\kappa_{p_1^5}&=0\\
8718\kappa_{e^6}-27635\kappa_{e^4p_1^2}+12842\kappa_{e^2p_1^4}-1382\kappa_{p_1^6}&=0\\
-7978\kappa_{e^6p_1}+11880\kappa_{e^4p_1^3}-4322\kappa_{e^2p_1^5}+420\kappa_{p_1^7}&=0\\
-68435\kappa_{e^8}+423040\kappa_{e^6p_1^2}-407726\kappa_{e^4p_1^4}+122508\kappa_{e^2p_1^6}-10851\kappa_{p_1^8}&=0\\
11098737\kappa_{e^8 p_1}-29509334\kappa_{e^6 p_1^3}+20996751\kappa_{e^4 p_1^5}-5391213\kappa_{e^2 p_1^7}+438670\kappa_{p_1^9}&=0.
\end{align*}
The first Hirzebruch relation allows us to remove $\kappa_{e^2}$ from the list of generators. The second Hirzebruch relation, with the relations $\kappa_{p_1^3} \equiv \tfrac{3}{2}\kappa_{ep_1^2} \equiv  3 \kappa_{e^2p_1}$ proved above, shows that $\kappa_{p_1^3}$ is decomposable. This proves the

\begin{lemma}\label{lem:CP2Gen}
The classes $\kappa_{p_1^2}, \kappa_{p_1^4}, \kappa_{ep_1}$ generate $R^*(\bC\bP^2)$.
\end{lemma}

Let the ideal $I$ of $\bQ[\kappa_{p_1^2}, \kappa_{p_1^4}, \kappa_{ep_1}]$ be generated by those relations implied by \eqref{eq:CP2rele}--\eqref{eq:CP2rel3} for $\kappa_{e^a p_1^b}$ for $a+b \leq 9$, and the Hirzebruch relations listed above.\footnote{The threshold $a+ b \leq 9$ is not significant, and one could try to go further, but we have checked that adding those relations with $a+b=10$ does not change the ideal $I$.} Generators for this ideal can be computed to be
\begin{align*}
&(4 \kappa_{p_1^2}-7 \kappa_{e p_1})(\kappa_{p_1^2}-2 \kappa_{e p_1}) \kappa_{p_1^4}\\
&(4 \kappa_{p_1^2}-7 \kappa_{e p_1})(\kappa_{p_1^2}-2 \kappa_{e p_1}) (21 \kappa_{e p_1}+8 \kappa_{p_1^2})\\
&(4 \kappa_{p_1^2}-7 \kappa_{e p_1})(316 \kappa_{e p_1}^3-343 \kappa_{p_1^4})\\
&(4 \kappa_{p_1^2}-7 \kappa_{e p_1}) (1264 \kappa_{p_1^2} \kappa_{e p_1}^2+2212 \kappa_{e p_1}^3-5145 \kappa_{p_1^4}).
\end{align*}
This ideal is not radical, and $\sqrt{I}$ is generated by
\begin{align*}
&(4 \kappa_{p_1^2}-7 \kappa_{e p_1})(\kappa_{p_1^2}-2 \kappa_{e p_1})\\
&(4 \kappa_{p_1^2}-7 \kappa_{e p_1})(316 \kappa_{e p_1}^3-343 \kappa_{p_1^4}).
\end{align*}

\begin{corollary}\label{cor:CP2Variety}
There is a surjection from
$$\bQ[\kappa_{p_1^2}, \kappa_{ep_1}, \kappa_{p_1^4}]/((4 \kappa_{p_1^2}-7 \kappa_{e p_1})(\kappa_{p_1^2}-2 \kappa_{e p_1}),(4 \kappa_{p_1^2}-7 \kappa_{e p_1})(316 \kappa_{e p_1}^3-343 \kappa_{p_1^4}))$$
to $R^*(\bC\bP^2)/\sqrt{0}$.
\end{corollary}

One can see that this quotient ring contains $\bQ[\kappa_{p_1^2}, \kappa_{p_1^4}]$ as a subring and is integral over it, so it has Krull dimension 2. It follows that $\Kdim(R^*(\bC\bP^2)) \leq 2$.

\subsubsection{Fixing a point}\label{sec:CP2FixPt}

It follows from Lemma \ref{lem:CP2Gen} that $R^*(\bC\bP^2, *)$ is generated by $e, p_1, \kappa_{p_1^2}, \kappa_{p_1^4}$ and $\kappa_{ep_1}$. Adding to the ideal $I$ above the relations \eqref{eq:CP2rele}--\eqref{eq:CP2rel3} gives an ideal $J$ of $\bQ[e, p_1, \kappa_{p_1^2}, \kappa_{p_1^4}, \kappa_{ep_1}]$ which is rather complicated, but its radical is generated by the relations
\begin{align*}
&(4 \kappa_{p_1^2}-7 \kappa_{e p_1})(\kappa_{p_1^2}-2 \kappa_{e p_1})\\
&1264 \kappa_{p_1^2} \kappa_{e p_1}^3-2212 \kappa_{e p_1}^4-1372 \kappa_{p_1^2} \kappa_{p_1^4}+2401 \kappa_{p_1^4} \kappa_{e p_1}\\
&10 \kappa_{p_1^2} \kappa_{e p_1}-28 \kappa_{p_1^2} p_1-21 \kappa_{e p_1}^2-14 \kappa_{e p_1} e+63 \kappa_{e p_1} p_1\\
&3 \kappa_{p_1^2} \kappa_{e p_1}-28 \kappa_{p_1^2} e-7 \kappa_{e p_1}^2+42 \kappa_{e p_1} e+7 \kappa_{e p_1} p_1\\
&45 \kappa_{p_1^2} \kappa_{e p_1}-112 \kappa_{e p_1}^2-84 \kappa_{e p_1} e+182 \kappa_{e p_1} p_1+196 e^2-196 p_1^2\\
&15 \kappa_{p_1^2} \kappa_{e p_1}-35 \kappa_{e p_1}^2+14 \kappa_{e p_1} e+35 \kappa_{e p_1} p_1+196 e^2-196 e p_1\\
&316 \kappa_{e p_1}^4+1264 \kappa_{e p_1}^3 e-1264 \kappa_{e p_1}^3 p_1-343 \kappa_{p_1^4} \kappa_{e p_1}-1372 \kappa_{p_1^4} e+1372 \kappa_{p_1^4} p_1\\ 
&12263 \kappa_{p_1^2} \kappa_{e p_1}^2-19446 \kappa_{e p_1}^3+168 \kappa_{e p_1}^2 e-168 \kappa_{e p_1}^2 p_1-4116 \kappa_{e p_1} e^2+16464 e^3-5488 \kappa_{p_1^4}
\end{align*}
the last of which shows that the generator $\kappa_{p_1^4}$ may be eliminated from the ring $\bQ[e, p_1, \kappa_{p_1^2}, \kappa_{p_1^4}, \kappa_{ep_1}]/\sqrt{J}$. One may also deduce from these relations that $\kappa_{ep_1}$ and $\kappa_{p_1^2}$ are integral over $\bQ[e, p_1]$, so that $R^*(\bC\bP^2, *)/\sqrt{0}$ is finite over $\bQ[e, p_1]$. 

\subsubsection{Fixing a disc}

As passing from $R^*(\bC\bP^2, *)$ to $R^*(\bC\bP^2, D^4)$ in particular kills $e$ and $p_1$, we deduce from the above that

\begin{corollary}
$R^*(\bC\bP^2, D^4)$ is a finite-dimensional $\bQ$-vector space.
\end{corollary}

In fact, setting $K=J+(e, p_1)$ and simplifying, we find that $K$ is generated by
\begin{align*}
\kappa_{p_1^2}^4 &\quad\quad\quad \kappa_{p_1^2}^2(105\kappa_{e p_1} -11\kappa_{p_1^2})\\
 \kappa_{p_1^2}(245\kappa_{e p_1}^2 -52\kappa_{p_1^2}^2) &\quad\quad\quad 1029\kappa_{e p_1}^3 -52\kappa_{p_1^2}^3\\
 245\kappa_{p_1^4}-29\kappa_{p_1^2}^3
\end{align*}
\begin{comment}
\begin{align*}
\kappa_{p_1^4}^2 & \quad\quad\quad
\kappa_{p_1^2}\kappa_{p_1^4}\\
\kappa_{e p_1}\kappa_{p_1^4} & \quad\quad\quad
29\kappa_{p_1^2}^3-245\kappa_{p_1^4}\\
609\kappa_{e p_1}^3-260\kappa_{p_1^4} & \quad\quad\quad
29\kappa_{p_1^2}\kappa_{e p_1}^2-52\kappa_{p_1^4}\\
87\kappa_{p_1^2}^2\kappa_{e p_1}-77\kappa_{p_1^4}
\end{align*}
\end{comment}
and $R^*(\bC\bP^2, D^4)$ is a quotient of $\bQ[\kappa_{p_1^2}, \kappa_{p_1^4}, \kappa_{ep_1}]/K$ so

\begin{corollary}
$\dim_\bQ R^*(\bC\bP^2, D^4) \leq 7$.
\end{corollary}
\begin{comment}
\begin{proof}
The last relation can be used to eliminate the generator $\kappa_{p_1^4}$, so we find that $R^*(\bC\bP^2, D^4)$ is a quotient of $\bQ[\kappa_{p_1^2}, \kappa_{e p_1}]$ by the ideal generated by
\begin{align*}
\kappa_{p_1^2}^4 &\quad\quad\quad 105\kappa_{p_1^2}^2\kappa_{e p_1} -11\kappa_{p_1^2}^3\\
245\kappa_{p_1^2}\kappa_{e p_1}^2 -52\kappa_{p_1^2}^3&\quad\quad\quad 1029\kappa_{e p_1}^3 -52\kappa_{p_1^2}^3
\end{align*}
\end{proof}
\end{comment}

\subsubsection{Lower bounds via torus actions}\label{sec:CP2Torus}

Consider the standard toric action of the torus $T=(S^1)^2$ on $\bC\bP^2$, via
\begin{align*}
S^1 \times S^1 \times \bC\bP^2 &\lra \bC\bP^2\\
(\xi_1,\xi_2,[z_0 : z_1 : z_2]) &\longmapsto [z_0 : \xi_1  z_1 : \xi_2  z_2]
\end{align*}
This gives a ring homomorphism $\phi : R^*(\bC\bP^2) \to H^*_T = \bQ[x_1, x_2]$. It is an elementary exercise to compute, by equivariant localisation, the classes
\begin{align*}
\phi(\kappa_{p_1^2}) &=7x_1^2-7x_1x_2+7x_2^2\\
\phi(\kappa_{ep_1}) &= 4x_1^2-4x_1x_2+4x_2^2\\
\phi(\kappa_{p_1^4}) &= 23x_1^6-69x_1^5x_2+135x_1^4x_2^2-155x_1^3x_2^3+135x_1^2x_2^4-69x_1x_2^5+23x_2^6
\end{align*}
and by eliminating variables to find that the unique relation between these is $\phi(7\kappa_{ep_1} - 4\kappa_{p_1^2})=0$. Thus $\phi$ gives a surjection
$$R^*(\bC\bP^2)/\sqrt{0} \lra \bQ[\kappa_{p_1^2}, \kappa_{ep_1}, \kappa_{p_1^4}]/(7\kappa_{ep_1} - 4\kappa_{p_1^2})$$
and hence $\Kdim(R^*(\bC\bP^2)) \geq 2$. Combining this with the above gives

\begin{corollary}
$\Kdim(R^*(\bC\bP^2)) = 2$.
\end{corollary}

The fixed point $[1:0:0]$ of the $T$-action gives an extension of $\phi$ to a ring homomorphism $\hat{\phi} :  R^*(\bC\bP^2, *)/\sqrt{0} \to H^*_T = \bQ[x_1, x_2]$. At this fixed point we have
\begin{align*}
\hat{\phi}(s^*e) &= x_1 x_2\\
\hat{\phi}(s^*p_1) &= x_1^2 + x_2^2
\end{align*}
which shows that the image of $\hat{\phi}$ is isomorphic to
$$\bQ[\kappa_{p_1^2}, \kappa_{ep_1}, \kappa_{p_1^4}, e, p_1]/(\kappa_{p_1^2}-7p_1+7e, \kappa_{ep_1}-4p_1+4e, 17e^3-66e^2p_1+69e p_1^2-23p_1^3+\kappa_{p_1^4}),$$
or in other words $\bQ[e, p_1]$.

\begin{remark}\label{rem:NoMoreCircles}
In \cite{Fintushel1, Fintushel2} there is given an analysis of $S^1$-actions on simply-connected 4-manifolds, from which it is possible to deduce---through a very laborious consideration of cases and analysis of fixed-point data---that for any circle action on $\bC\bP^2$ we have $4\kappa_{e^2} = \kappa_{ep_1}$ and so by the first Hirzebruch relation we have $4 \kappa_{p_1^2}-7 \kappa_{e p_1}=0$. Alternatively, this may be proved using Hsiang's splitting theorem for the $S^1$-equivariant cohomology of $\bC\bP^2$ \cite[Theorem VI.1]{Hsiang}. 
\end{remark}

\subsubsection{The tautological variety}

We find it quite revealing to consider the (reduced) tautological ring $R^*(\bC\bP^2)/\sqrt{0}$ by considering its associated variety $\gV_{\bC\bP^2}$. The choice of generators $\kappa_{p_1^2}, \kappa_{p_1^4},$ and $\kappa_{ep_1}$ of $R^*(\bC\bP^2)$ presents $\gV_{\bC\bP^2}$ as a subvariety of $\gA^3$, and it follows from Corollary \ref{cor:CP2Variety} that $\gV_{\bC\bP^2}$ is contained in the union of the plane
$$\gP :=\{4 \kappa_{p_1^2}-7 \kappa_{e p_1}=0\}$$
and the line
$$\gL :=\{\kappa_{p_1^2}-2 \kappa_{e p_1}=0, 316 \kappa_{e p_1}^3-343 \kappa_{p_1^4}=0\}.$$
Furthermore, it follows from the calculation of Section \ref{sec:CP2Torus} that $\gV_{\bC\bP^2}$ contains $\gP$, so the variety $\gV_{\bC\bP^2}$ is either $\gP$ or $\gP \cup \gL$. It would be extremely interesting if $\gL \subset \gV_{\bC\bP^2}$, but no method for showing this seems to be available. (Each circle action on $\bC\bP^2$ gives a homomorphism $R^*(\bC\bP^2)/\sqrt{0} \to \bQ[x_1]$ and hence a morphism $\gA^1 \to \gV_{\bC\bP^2}$, but by Remark \ref{rem:NoMoreCircles} all such morphisms have image in $\gP$.)

Similarly, by the calculation of Section \ref{sec:CP2FixPt} the four elements $e, p_1, \kappa_{p_1^2}$, and $\kappa_{ep_1}$ generate $R^*(\bC\bP^2,*)/\sqrt{0}$, which presents the associated variety $\gV_{(\bC\bP^2,*)}$ as a subvariety of $\gA^4$. Eliminating the variable $\kappa_{p_1^4}$ from the radical ideal described in Section \ref{sec:CP2FixPt} shows that $\gV_{(\bC\bP^2,*)}$ is contained in the union of the plane
$$\{4 \kappa_{p_1^2}-7 \kappa_{e p_1}=0, \kappa_{ep_1} -4p_1 + 4e=0\}$$
and the lines
\begin{align*}
\{\kappa_{p_1^2}-2 \kappa_{e p_1}=0, e=0, \kappa_{ep_1}-7p_1 &=0\}\\
\{\kappa_{p_1^2}-2 \kappa_{e p_1}=0, 2\kappa_{ep_1}-7e=0, 5\kappa_{ep_1}-7p_1 &=0\}.
\end{align*}
It follows from the calculation of Section \ref{sec:CP2Torus} that the plane is contained in $\gV_{(\bC\bP^2,*)}$.

\subsection{The manifold $S^2 \times S^2$}\label{sec:exS2S2}

The cohomology of $S^2 \times S^2$ is supported in even degrees. Thus by Theorem \ref{mainThm:A} the algebra $R^*(S^2 \times S^2)$ is finitely-generated and $R^*(S^2 \times S^2, *)$ is a finite $R^*(S^2 \times S^2)$-module. 

The trace identity technique of Section \ref{sec:TraceId} gives the relation
\begin{align*}
c^4&=\kappa_{ec}c^3-\frac{\kappa_{ec}^2-\kappa_{ec^2}}{2}c^2+\frac{\kappa_{ec}^3-3\kappa_{ec}\kappa_{ec^2}+2\kappa_{ec^3}}{6}c\\
&\quad-\frac{\kappa_{ec}^4}{24}+\frac{\kappa_{ec}^2\kappa_{ec^2}}{4}-\frac{\kappa_{ec^2}^2}{8}-\frac{\kappa_{ec}\kappa_{ec^3}}{3}+\frac{\kappa_{ec^4}}{4} \in R^*(S^2 \times S^2, *)
\end{align*}
for any $c \in H^*(BSO(4))=\bQ[p_1, e]$. Partially polarising via $c=e+t\cdot p_1$ as in Section \ref{sec:exCP2}, we obtain the relations
\begin{align}
&(1/8) \kappa_{ep_1^2}^2+(1/24) \kappa_{ep_1}^4-(1/6) p_1 \kappa_{ep_1}^3-(1/4) \kappa_{ep_1}^2 \kappa_{ep_1^2}\nonumber\\
&\quad+(1/2) p_1^2 \kappa_{ep_1}^2-(1/3) p_1 \kappa_{ep_1^3}+(1/3) \kappa_{ep_1} \kappa_{ep_1^3}-p_1^3 \kappa_{ep_1}\nonumber\\
&\quad-(1/2) p_1^2 \kappa_{ep_1^2}-(1/4) \kappa_{ep_1^4}+p_1^4+(1/2) p_1 \kappa_{ep_1} \kappa_{ep_1^2}=0\label{eq:S2S2:1}\\
&-p_1^2 \kappa_{e^2p_1}-(1/6) \kappa_{ep_1}^3 e+\kappa_{ep_1} \kappa_{e^2p_1^2}-p_1 \kappa_{e^2p_1^2}+(1/6) \kappa_{ep_1}^3 \kappa_{e^2}\nonumber\\
&\quad+(1/3) \kappa_{ep_1^3} \kappa_{e^2}-p_1^3 \kappa_{e^2}+(1/2) \kappa_{ep_1^2} \kappa_{e^2p_1}-(1/2) \kappa_{ep_1}^2 \kappa_{e^2p_1}\nonumber\\
&\quad+4 e p_1^3-(1/3) e \kappa_{ep_1^3}-(1/2) \kappa_{ep_1} \kappa_{ep_1^2} \kappa_{e^2}+(1/2) p_1 \kappa_{ep_1^2} \kappa_{e^2}\nonumber\\
&\quad-(1/2) p_1 \kappa_{ep_1}^2 \kappa_{e^2}+p_1^2 \kappa_{ep_1} \kappa_{e^2}-e p_1 \kappa_{ep_1^2}+(1/2) e \kappa_{ep_1} \kappa_{ep_1^2}\nonumber\\
&\quad+e p_1 \kappa_{ep_1}^2-3 e p_1^2 \kappa_{ep_1}+p_1 \kappa_{ep_1} \kappa_{e^2p_1}-\kappa_{e^2p_1^3}=0\label{eq:S2S2:2}\\
&6 e^2 p_1^2+(1/4) \kappa_{ep_1}^2 \kappa_{e^2}^2+(1/2) p_1^2 \kappa_{e^2}^2-(1/2) e^2 \kappa_{ep_1^2}\nonumber\\
&\quad+(1/2) e^2 \kappa_{ep_1}^2-p_1 \kappa_{e^3p_1}+\kappa_{ep_1} \kappa_{e^3p_1}+(1/4) \kappa_{ep_1^2} \kappa_{e^3}-(1/4) \kappa_{ep_1}^2 \kappa_{e^3}\nonumber\\
&\quad-(1/2) p_1^2 \kappa_{e^3}-e \kappa_{e^2p_1^2}+\kappa_{e^2} \kappa_{e^2p_1^2}-(1/4) \kappa_{ep_1^2} \kappa_{e^2}^2\nonumber\\
&\quad-(1/2) e \kappa_{ep_1}^2 \kappa_{e^2}-3 e p_1^2 \kappa_{e^2}-3 e^2 p_1 \kappa_{ep_1}\nonumber\\
&\quad+(1/2) p_1 \kappa_{ep_1} \kappa_{e^3}+(1/2) \kappa_{e^2p_1}^2+(1/2) e \kappa_{ep_1^2} \kappa_{e^2}\nonumber\\
&\quad-(1/2) p_1 \kappa_{ep_1} \kappa_{e^2}^2+e \kappa_{ep_1} \kappa_{e^2p_1}+p_1 \kappa_{e^2} \kappa_{e^2p_1}\nonumber\\
&\quad-2 e p_1 \kappa_{e^2p_1}-\kappa_{ep_1} \kappa_{e^2} \kappa_{e^2p_1}+2 e p_1 \kappa_{ep_1} \kappa_{e^2}-(3/2) \kappa_{e^3p_1^2}=0\label{eq:S2S2:3}\\
&(1/2) \kappa_{e^2 p_1} \kappa_{e^3}-(1/2) \kappa_{e^2}^2 \kappa_{e^2 p_1}-e^2 \kappa_{e^2 p_1}-(1/6) \kappa_{e^2}^3 p_1\nonumber\\
&\quad+(1/6) \kappa_{e^2}^3 \kappa_{e p_1}+\kappa_{e^2} \kappa_{e^3 p_1}-e^3 \kappa_{e p_1}+(1/3) \kappa_{e p_1} \kappa_{e^4}-(1/3) p_1 \kappa_{e^4}\nonumber\\
&\quad-e \kappa_{e^3 p_1}+(1/2) p_1 \kappa_{e^2} \kappa_{e^3}-(1/2) \kappa_{e p_1} \kappa_{e^2} \kappa_{e^3}+(1/2) e \kappa_{e p_1} \kappa_{e^3}\nonumber\\
&\quad-e p_1 \kappa_{e^3}+e \kappa_{e^2} \kappa_{e^2 p_1}+e p_1 \kappa_{e^2}^2-(1/2) e \kappa_{e p_1} \kappa_{e^2}^2-3 e^2 p_1 \kappa_{e^2}\nonumber\\
&\quad+e^2 \kappa_{e p_1} \kappa_{e^2}-\kappa_{e^4 p_1}+4 e^3 p_1=0\label{eq:S2S2:4}\\
&(1/8) \kappa_{e^3}^2+(1/24) \kappa_{e^2}^4-(1/6) e \kappa_{e^2}^3+(1/2) e^2 \kappa_{e^2}^2-(1/4) \kappa_{e^2}^2 \kappa_{e^3}\nonumber\\
&\quad+(1/2) e \kappa_{e^2} \kappa_{e^3}-(1/4) \kappa_{e^5}-e^3 \kappa_{e^2}-(1/2) e^2 \kappa_{e^3}-(1/3) e \kappa_{e^4}\nonumber\\
&\quad+(1/3) \kappa_{e^2} \kappa_{e^4}+e^4=0\label{eq:S2S2:5}
\end{align}
Modulo decomposables in $R^*(S^2 \times S^2)$, when multiplied by monomials in $\bQ[e, p_1]$ and fibre integrated these give the relations
\begin{align*}
\kappa_{x p_1^4} &\text{ is decomposable for any monomial $x \neq e$}\\
\kappa_{x ep_1^3} &\text{ is decomposable for any monomial $x \neq e$}\\
\kappa_{x e^2 p_1^2} & \text{ is decomposable for any monomial $x \neq e$}\\
\kappa_{x e^3 p_1} & \text{ is decomposable for any monomial $x \neq 1, e$}\\
\kappa_{x e^4} & \text{ is decomposable for any monomial $x \neq e$}.
\end{align*}
%Furthermore, multiplying the second relation above by $p_1$ and fibre integrating shows that $\kappa_{ep_1^4}$ is decomposable. 
This shows that all generators apart from
$$\kappa_{p_1^2}, \kappa_{p_1^3}, \kappa_{ep_1}, \kappa_{ep_1^2}, \kappa_{e^2}, \kappa_{e^2p_1}, \kappa_{e^3}, \kappa_{e^3p_1}, \kappa_{e^5}$$
are decomposable. The Hirzebruch relations of the previous section hold here as well, as the bilinear form associated to $S^2 \times S^2$ is $(\bZ^2, \left(\begin{smallmatrix} 0&1\\ 1&0 \end{smallmatrix}\right))$ which also has finite automorphism group. The first two Hirzebruch relations $\kappa_{e^2} = \frac{1}{7} \kappa_{p_1^2}$ and $\kappa_{e^2p_1} = \frac{2}{13}\kappa_{p_1^3}$ allow us to remove two of these generators, and so we find that

\begin{lemma}
The classes $\kappa_{p_1^2}, \kappa_{p_1^3}, \kappa_{e p_1}, \kappa_{ep_1^2}, \kappa_{e^3}, \kappa_{e^3 p_1}, \kappa_{e^5}$ generate $R^*(S^2 \times S^2)$. 
\end{lemma}

Consider the ideal $I$ of $\bQ[\kappa_{p_1^2}, \kappa_{p_1^3}, \kappa_{e p_1}, \kappa_{ep_1^2}, \kappa_{e^3}, \kappa_{e^3 p_1}, \kappa_{e^5}]$ of relations implied by \eqref{eq:S2S2:1}--\eqref{eq:S2S2:5} for $\kappa_{e^a p_1^b}$ for $a+b \leq 9$, and the Hirzebruch relations of the previous section. It is quite complicated, but it is easy to compute (in {\tt Macaulay2}) that it has codimension 3.

\begin{corollary}
$\Kdim(R^*(S^2 \times S^2)) \leq 4$.
\end{corollary}

\subsubsection{Fixing a point}

It follows from Lemma \ref{lem:CP2Gen} that $R^*(S^2 \times S^2, *)$ is generated by $e, p_1, \kappa_{p_1^2}, \kappa_{p_1^3}, \kappa_{e p_1}, \kappa_{ep_1^2}, \kappa_{e^3}, \kappa_{e^3 p_1}, \kappa_{e^5}$. Adding to the ideal $I$ above the relations \eqref{eq:S2S2:1}--\eqref{eq:S2S2:5} gives an ideal $J$ of $\bQ[e, p_1, \kappa_{p_1^2}, \kappa_{p_1^3}, \kappa_{e p_1}, \kappa_{ep_1^2}, \kappa_{e^3}, \kappa_{e^3 p_1}, \kappa_{e^5}]$ which has codimension 5.

\subsubsection{Fixing a disc}

Passing from $R^*(S^2 \times S^2, *)$ to $R^*(S^2 \times S^2, D^4)$ in particular kills $e$ and $p_1$, and we may compute the radical of the ideal $K:=J+(e, p_1)$, giving the following.

\begin{corollary}
$R^*(S^2 \times S^2, D^4)/\sqrt{0}$ is a quotient of
$$\frac{\bQ[\kappa_{p_1^2}, \kappa_{p_1^3},\kappa_{e p_1}, \kappa_{ep_1^2}, \kappa_{e^3}, \kappa_{e^3 p_1}, \kappa_{e^5}]}{(\kappa_{p_1^3},\kappa_{p_1^2},\kappa_{e^3}^2-2\kappa_{e^5},\kappa_{e p_1}\kappa_{e^3}-2\kappa_{e^3 p_1},\kappa_{e p_1}^2-2\kappa_{ep_1^2})} \cong \bQ[\kappa_{e p_1}, \kappa_{e^3}]$$
so has Krull dimension at most 2.
\end{corollary}

\subsubsection{Lower bounds via torus actions}

We will use a family of almost-complex torus actions $\phi_k : T^2 \to \Diff(S^2 \times S^2)$ defined for $k \in \bN$. % there is an which has four fixed points, with weights
%$$\{-x_2, x_1-2kx_2\}, \{-x_2, 2kx_2-x_1\}, \{x_2, x_1\}, \{x_2, -x_1\}.$$
These actions are well-known among symplectic geometers: we learnt their construction from work of Karshon \cite{Karshon}, suggested to us by Ivan Smith. In that paper these actions are constructed as the toric varieties associated to the Delzant polytopes
\begin{figure}[h]
\includegraphics[height=2.5cm]{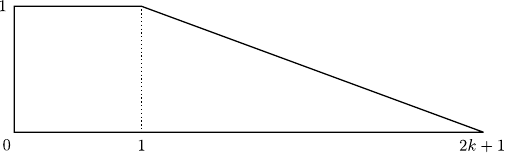}
\end{figure}

\noindent and it follows from \cite[Lemma 3]{Karshon}, and the fact that $k=0$ yields $S^2 \times S^2$, that all the manifolds so obtained are diffeomorphic to $S^2 \times S^2$. In toric geometry the above polytope should be considered as lying in the dual $\mathfrak{t}^*$ of the Lie algebra of $T$, having integral basis $\{x_1, x_2\}$ which we identify with the cartesian coordinates in the figure above. The $T^2$-fixed points correspond to the vertices of the polytope, and the weights at each fixed point are given by the pair of elements of $\mathfrak{t}^*$ given by the two primitive integral vectors associated to the edges incident at that vertex (cf.\ \cite[Example 7.3.19]{ToricTop}). For the polytope above the weights are therefore
$$\{x_1, x_2\}, \{x_1, -x_2\}, \{-x_1, 2kx_1-x_2\}, \{-x_1, x_2-2kx_1\}.$$

It follows that at the four fixed points of the action $\phi_k$ the Euler class is
$$x_1x_2, -x_1x_2, x_1(x_2-2kx_1), x_1(2kx_1-x_2)$$
%$$x_1(2kx_1-x_2), x_1(x_2-2kx_1), x_1x_2, -x_1x_2$$
and the Pontrjagin class $p_1 = c_1^2-2c_2$ is
$$x_1^2+x_2^2, x_1^2+x_2^2, (4k^2+1)x_2^2-4kx_1x_2+x_1^2, (4k^2+1)x_2^2-4kx_1x_2+x_1^2$$
%$$(4k^2+1)x_1^2-4kx_2x_1+x_2^2, (4k^2+1)x_1^2-4kx_2x_1+x_2^2, x_1^2+x_2^2, x_1^2+x_2^2.$$
We may thus compute the map
$$\psi_k : \bQ[\kappa_{p_1^2}, \kappa_{p_1^3}, \kappa_{e p_1}, \kappa_{ep_1^2}, \kappa_{e^3}, \kappa_{e^3 p_1}, \kappa_{e^5}] \lra R^*(S^2 \times S^2) \overset{\phi_k}\lra H^*_T = \bQ[x_1, x_2]$$
by equivariant localisation, giving
\begin{align*}
\kappa_{p_1^2} &= 0\\
\kappa_{p_1^3} &= 0\\
\kappa_{e p_1} &= 8 k^2 x_2^2-8 k x_2 x_1+4 x_2^2+4 x_1^2\\
\kappa_{ep_1^2} &= 32 k^4 x_2^4-64 k^3 x_2^3 x_1+16 k^2 x_2^4+48 k^2 x_2^2 x_1^2-16 k x_2^3 x_1-16 k x_2 x_1^3+4 x_2^4\\
&\,\,\,\,\,\,\,\,+8 x_2^2 x_1^2+4 x_1^4\\
\kappa_{e^3} &= 8 k^2 x_2^4-8 k x_2^3 x_1+4 x_2^2 x_1^2\\
\kappa_{e^3 p_1} &= 32 k^4 x_2^6-64 k^3 x_2^5 x_1+8 k^2 x_2^6+48 k^2 x_2^4 x_1^2-8 k x_2^5 x_1-16 k x_2^3 x_1^3+4 x_2^4 x_1^2\\
&\,\,\,\,\,\,\,\,+4 x_2^2 x_1^4\\
\kappa_{e^5} &= 32 k^4 x_2^8-64 k^3 x_2^7 x_1+48 k^2 x_2^6 x_1^2-16 k x_2^5 x_1^3+4 x_2^4 x_1^4
\end{align*}
By eliminating $x_1$, $x_2$, and $k$ from the above, one finds generators for the ideal $U:=\cap_{k \in \bN} \Ker(\psi_k)$ of $\bQ[\kappa_{p_1^2}, \kappa_{p_1^3}, \kappa_{e p_1}, \kappa_{ep_1^2}, \kappa_{e^3}, \kappa_{e^3 p_1}, \kappa_{e^5}]$ to be
\begin{align*}
& \kappa_{p_1^2}\\
& \kappa_{p_1^3}\\
& \kappa_{ep_1^2}\kappa_{e^3} - \kappa_{e^3}^2 - \kappa_{ep_1}\kappa_{e^3 p_1} + 4\kappa_{e^5}\\ & \kappa_{e^3}^3 - \kappa_{ep_1}\kappa_{e^3}\kappa_{e^3p_1} + \kappa_{ep_1}^2 \kappa_{e^5} + 4\kappa_{e^3 p_1}^2 - 4\kappa_{ep_1^2}\kappa_{e^5} - 4\kappa_{e^3}\kappa_{e^5},
\end{align*}
so this ideal has codimension 4. There is a surjection
$$R^*(S^2 \times S^2) \lra \bQ[\kappa_{p_1^2}, \kappa_{p_1^3}, \kappa_{e p_1}, \kappa_{ep_1^2}, \kappa_{e^3}, \kappa_{e^3 p_1}, \kappa_{e^5}]/U,$$
and hence the Krull dimension of $R^*(S^2 \times S^2)$ is bounded below by $7-\mathrm{codim}(U)=3 $.

\begin{corollary}
$\Kdim(R^*(S^2 \times S^2)) \geq 3$.
\end{corollary}

Note that each ideal $\Ker(\psi_k)$ of $\bQ[\kappa_{p_1^2}, \kappa_{p_1^3}, \kappa_{e p_1}, \kappa_{ep_1^2}, \kappa_{e^3}, \kappa_{e^3 p_1}, \kappa_{e^5}]$ has codimension 5, so each particular torus action only gives 2 as a lower bound for $\Kdim(R^*(S^2 \times S^2))$: it is only by considering the countably-many such actions that we are able to improve this lower bound to 3.

\bibliographystyle{amsalpha}
\bibliography{biblio}

\end{document}